\documentclass{article}
\usepackage{amsfonts}
\usepackage{amsmath}
\usepackage{amssymb}
\usepackage{amsbsy}
\usepackage{amsthm}
\usepackage{epsfig}
\usepackage{graphicx}
\usepackage{colordvi}
\usepackage{graphics}
\usepackage{color}
\usepackage{subfigure}
\numberwithin{equation}{section}
\newcommand{\norm}[1]{\left\Vert#1\right\Vert}
\newtheorem{theorem}{Theorem}[section]
\newtheorem{lemma}{Lemma}[section]
\newtheorem{definition}{Definition}[section]
\newtheorem{remark}{Remark}[section]

\newcommand{\Real}{\mathbb{R}}

\newcommand{\tore}{\mathbb{T}}
\newcommand{\R}{\mathbb{R}}

\newcommand{\N}{\mathbb{N}}
\newcommand{\Z}{\mathbb{Z}}
\newcommand{\C}{\mathbb{C}}

\begin{document}

\title{Analysis of a Reduced-Order Approximate Deconvolution Model and
  its interpretation as a Navier-Stokes-Voigt regularization}

\author{ Luigi C. Berselli\thanks{Dipartimento di Matematica,
    Universit\`a di Pisa, I-56127 Pisa, Italy
    (luigi.carlo.berselli@unipi.it), partially supported by GNAMPA and
    by PRIN 2012: 
    ``Nonlinear hyperbolic pdes dispersive and transport equations: theoretical and
  applicative aspects.''}
  \and
  Tae-Yeon Kim\thanks{Department of Civil Infrastructure and
    Environmental Engineering, Khalifa University of Science,
    Technology \& Research (KUSTAR), Abu Dhabi,
     UAE (taeyeon.kim@kustar.ac.ae)}
\and 
Leo G. Rebholz\thanks{Department of Mathematical Sciences, Clemson
  University, Clemson, SC 29634 (rebholz@clemson.edu), partially
  supported by NSF grant DMS1112593.} } 

\date{}
\maketitle
\begin{abstract}
  We study mathematical and physical properties of a family of
  recently introduced, reduced-order approximate deconvolution models.
  We first show a connection between these models and the NS-Voigt
  model, and 
  that NS-Voigt can be re-derived in the
  approximate deconvolution framework.  We then study the energy
  balance and spectra of the model, and provide results of some
  turbulent flow computations that backs up the theory.  Analysis of
  global attractors for the model is also provided, as is a detailed
  analysis of the Voigt model's treatment of pulsatile flow.
\end{abstract}
\section{Introduction}
Over the past several decades, many large eddy simulation (LES) models
have been introduced and tested, with varying degrees of success; see
e.g.~\cite{BIL06,SAG01} for extensive overviews.  Each of these models
has its share of good and bad properties, and with each generation of
models, the amount of ``good'' tends to increase while the amount of
``bad'' tends to decrease.  Two attractive models of the current
generation are approximate deconvolution models
(ADMs)~\cite{AS99,AS01}, and Voigt regularization
models~\cite{CLT2006,LT10}, both of which have recently seen
significant interest from both mathematicians and engineers.  These
models are known to be well-posed~\cite{BB12,BL2012,DE06,LT10}, have
attractive physical properties (such as energy and helicity
conservation and cascades~\cite{LT10,LR12}), and have performed well
in (at least some) numerical simulations~\cite{KLRW12,sto022},
including when applied to the Euler equations for ideal
fluids~\cite{CLT2006}.

However, both models have their downsides: It is unknown how to
efficiently and stably implement the usual ADM with standard finite
element methods, and the Voigt
regularization model can lack accuracy, likely due to its low
consistency error with both the Navier-Stokes equations (NSE) 
\begin{eqnarray} 
  \label{eq:NSE}
  {u}_t+{(u\cdot\nabla)\,u}-\nu \Delta
  u+\nabla p&=&f,
  \\
  \nabla \cdot u & = & 0, \label{eq:NSE2} 
\end{eqnarray}
and the spatially filtered NSE~\cite{GRT13,LRT10}. Here $u$ denotes
the fluid velocity and $p$ the kinematic pressure.

In the recent papers~\cite{GRT13,R13b}, a new system was proposed as
an alteration of the usual ADM, but which is amenable to
unconditionally stable and efficient implementation in a finite
element context. Here, we add further insight in the derivation, from
an abstract point of view, see Section~\ref{sec:derivation}.  This
model, which we study herein, is referred to as the reduced-order ADM
(RADM), and is given by
\begin{eqnarray}
  v_t-\alpha^2\Delta v_t+D v\cdot\nabla D v
  +\nabla q-\nu\Delta D v  &=&f,\label{eq:cadm1}
  \\
  \nabla \cdot Dv & = & 0, \label{eq:cadm2} 
\end{eqnarray}
where the vector $v$ represents an averaged velocity, $q$ a modified
pressure, while $D$ is an approximate deconvolution operator meant to
represent an approximate inverse to the Helmholtz filter $F$, with
$F\phi:=\overline{\phi}$ satisfying the filter equation:
\begin{equation*}
  -\alpha^2\Delta \overline{\phi} + \overline{\phi} = \phi.
\end{equation*}
The evolution system~\eqref{eq:cadm1}-\eqref{eq:cadm2} requires an
initial condition for velocity, and both the filter and evolution
systems require appropriate boundary conditions.
By restricting to the space-periodic setting (as usual for these
problems, with only a very few remarkable exceptions), we can assume
$D$ to be linear, positive, self-adjoint, to commute with $F$, and
moreover that there exists $d_0,d_1\in\R$ such that
\begin{equation*}
  1\leq d_0\leq\|D\|\leq d_1<\infty,
\end{equation*}
where $\|D\|$ denotes the operator norm of $D$. These assumptions are
true for most common types of approximate deconvolution defined with
the Helmholtz filtering operation, including van Cittert~\cite{DE06},
multiscale~\cite{D11}, Tikhonov, and Tikhonov-Lavrentiev~\cite{SM10},
see also the review in~\cite{Ber2012}. While we will often be general
in our use of $D$, in some circumstances such as computations and
energy spectra calculations, it is necessary to select a specific
operator.  In these cases, we will use the most common type of
approximate deconvolution, van Cittert, and we will denote $D=D_N$,
where $N$ denotes the order of deconvolution.  This operator is
defined in detail in Section~\ref{sec:notation}.

The RADM system~\eqref{eq:cadm1}-\eqref{eq:cadm2} was derived from the
fourth-order version of the ADM by making the approximation $D\approx
F^{-1}$ in the viscous term, which reduced the fourth order system to
second order.  This system was proved to be well posed in~\cite{R13b}
and of high order consistency with the spatially filtered NSE.
In~\cite{GRT13}, with $D$ chosen to be van Cittert approximate
deconvolution, an efficient finite element implementation of the RADM
performed very well on the $Re_{\tau}=180$ turbulent channel flow
benchmark test, as well as in a 2D benchmark flow problem with a
contraction and two outlets.

\textbf{Plan of the Paper.}
It is the purpose of this paper to study some fundamental mathematical
and physical properties of the model~\eqref{eq:cadm1}-\eqref{eq:cadm2}.
After providing mathematical preliminaries in
Section~\ref{sec:notation}, we show a strong connection between the
RADM~\eqref{eq:cadm1} and the NS-Voigt model~\eqref{eq:NS-Voigt} in 
Section~\ref{sec:derivation}.  In fact, we show that the RADM can be
considered as a NS-Voigt-deconvolution model.  That is, if the
regularization operator $F^{-1}$ is replaced with $D^{-1}F^{-1}$, then
the RADM is recovered after a change of variables.  This is
interesting and important since a significant amount of theory has
been developed for the NS-Voigt model by Titi and coworkers in
e.g.~\cite{KT09,LRT10,LT10} as well as in~\cite{BB12}.  We also show
that NS-Voigt can be derived using the approximate deconvolution
approach originally developed by Stolz and
Adams~\cite{AS99,AS01,ada021,sto022}.  In this respect we establish a
strong tie between Voigt-type models and approximate deconvolution type
models in general.

An energy analysis of the RADM is given in Section~4.  This includes
an energy balance and a Kraichnan-type analysis of the energy spectra
of the model.  Some numerical results are also given here, to back up
the theory.  Our results reveal that the RADM is more computable than
the NSE and has spectral scaling similar to the Leray-$\alpha$
model~\cite{CHOT05}, which matches that of true fluid flow on larger
scales in the inertial range, but smaller scales are transferred more
quickly to even smaller scales. Moreover, our analysis shows that increasing the
deconvolution order $N$ shortens the inertial range and creates less
energy in each scale.  

Section 5 considers long-time behavior of the model, and proves existence of smooth
global attractors.   Section 6 considers the
pseudo-parabolic nature of the system, and presents a study of
some analytical properties and the model's treatment of pulsatile
flow, which is an interesting example because it allows for exact
analytical solutions.  Finally, conclusions are drawn in the final
section, and prospective future work is discussed.

\medskip

%
\section{Preliminaries}
\label{sec:notation}
We present here notation and mathematical preliminaries to make the
analysis of later sections smoother.  As is typical with theoretical
analyses of fluid models, we restrict our attention herein to problems
involving periodic flows on the box $\Omega = (0,2\pi L)^3$.  We
stress that all machinery requires the space-periodic setting to have
the commutation between the deconvolution operator and general
differential operators. The application of ADM methods to motion
bounded by rigid walls is problematic and the theory is still
missing. Some results in special geometries can be found
in~\cite{Ber2012b,sto022}. In this framework, NS-Voigt
represents one of the few models which is naturally set in a bounded
domain with Dirichlet conditions.

Furthermore, we consider the system~\eqref{eq:cadm1}-\eqref{eq:cadm2} to be
non-dimensionalized by introducing $L$, $U$, and $L/U$ for the
characteristic length, velocity, and time scales.  Thus, the viscosity
$\nu>0$ can be considered dimensionless and representing the inverse
of the Reynolds number of the flow. To simplify the notation, in the
theorems we will suppose without loss of generality that $L=1$.

We characterize the divergence-free spaces by using Fourier Series on
the 3D torus. We denote by $({e}_1, { e}_2, { e}_3)$ the orthonormal
basis of $\R^3$, and by $x:=(x_1, x_2, x_3) \in \R^3$ the standard
point in $\R^3$.  Let $\mathbb{T}$ be the torus defined by $\mathbb{T}
:= \R^3 / \Z^3$.  We use $\|\cdot\|$ and $(\cdot,\cdot)$ to denote the
$L^{2}(\tore)$ norm and inner product, and we impose the zero
mean-value condition on velocity, pressure, and external force.  We
will use the customary Lebesgue $L^p(\tore)$ and Sobolev
$W^{k,p}(\tore)$ spaces, and for simplicity (if not explicitly needed)
we do not distinguish between scalar and vector valued real functions.

We define, for any exponent $s\geq0$,
\begin{equation*}
  V^s := \left\{w : \tore \rightarrow \R^3, \, \,
    w \in W^{s,2}(\tore)^3,   \quad\nabla\cdot w
    = 0, \quad\int_{\tore}w(x)\,dx = {0} \right\},
\end{equation*}
where $W^{s,2}(\tore)^3:=\big[W^{s,2}(\tore)\big]^3$ and if $0\leq
s<1$ the condition $\nabla\cdot w=0$ must be understood in a weak
sense. For $w \in V^s$, we can expand the velocity field with Fourier
series as follows
\begin{equation*}
  w (x)=\sum_{|k|\geq1}\widehat{w}_{k}
  e^{i{k \cdot x}},\quad\text{where }k=(k_1,k_2,k_3)\in \Z^{3} \text{
    is the wave-number vector,}
\end{equation*}
and the Fourier coefficients are given by
$\C^3\ni\widehat{w}_{k}:=\frac{1}{|\tore |}\int_{\tore} w(x)e^{-i{k
    \cdot x}}dx$, with $|\tore |$ the Lebesgue measure of $\tore$. If
$|k|:=\sqrt{|k_{1}|^{2}+|k_{2}|^{2}+|k_{3}|^2}$, then the $V^s$ norm
is defined by
\begin{equation*}
  \| w \|^{2}_{V^s} := \sum_{|k|\geq1} |k|^{2s} |\widehat{w }_{k}|^{2},
\end{equation*}
where, as above, $\| w \|_{H^0} := \|w\|$. The inner products
associated with these norms are
\begin{equation*}
  (w, v )_{V^s} = \sum_{|k|\geq1} |k|^{2s}  \widehat{w
  }_{k}\cdot\overline{\widehat{v }_{k}}. 
\end{equation*}
We can finally characterize $V^s\subset W^{s,2}(\tore)^3$ as follows:
\begin{equation*}
  V^s := \Big \{ w(x)  = \sum_{{|k|}\geq1} \widehat w_{ k}e^{i{k \cdot x}}:\ \  
  \sum_{{|k|\geq1}} | {k} |^{2s} |\widehat{w}_{ k}|^{2} < \infty,\ \ 
  {k}\cdot \widehat{w}_{k}=0,\ \ 
  \widehat{w}_{-k}= \overline{\widehat{w}_{k}}\  \Big\}.
\end{equation*}
Recall that due to the zero mean value, the Poincar\'e inequality
holds in $V^1$: There exists $\lambda>0$ such that for every $\phi\in
V^1$,
\begin{equation*}
\norm{\phi}\le \lambda^{-1/2}\norm{\nabla \phi}.
\end{equation*}
This inequality will be used extensively in our analysis, and allows
us to use the more convenient gradient norm in place of the $V^1$
norm.
\subsection{Approximate deconvolution}
\label{sec:app-dec}
We assume the deconvolution operator $D$ is self adjoint  and
commutes with both $\Delta$ and $F$. These assumptions are known to be
satisfied for van Cittert approximate deconvolution~\cite{DE06,LR12},
multiscale deconvolution~\cite{D11}, Tikhonov, and
Tikhonov-Lavrentiev~\cite{SM10}. See also the reviews
in~\cite{Ber2012,CL14} for more results about these operators.

We denote the filtering operator $A:=F^{-1}=(-\alpha^2\Delta + I)$,
and we introduce now a family of abstract deconvolution operators
$\{D_N\}_N$, which are of order zero, and such that
\begin{equation*}
  D_0=I\quad\text{and }\quad D_N\to A \text{ as }N\to+\infty.
\end{equation*}
We note that van Cittert approximate deconvolution,
\begin{equation*}
D_N = \sum_{n=0}^N (I-F)^n,
\end{equation*}
forms such a family, and the other types of deconvolution mentioned
above can also be written to satisfy the above limiting
property.  Although we will use a general deconvolution operator $D$
when possible, in some analysis (e.g. energy spectra analysis) and for
computations, it is necessary to specify a specific operator.  In
these cases, we will use van Cittert approximate deconvolution, as it
is one of the most popular in the ADM community. The basic properties
we will use of the operator $D=D_N$ are summarized in the following
lemma, see~\cite{BL2012}.
\begin{lemma}
  \label{lem:lower_bound}
  For each $N\in\N$ the operator $D_N:\ V^s \to V^s$ is self-adjoint, it
  commutes with differentiation, and the following properties hold true:
  If $\widehat{D}_N$ denotes the symbol of the operator $D_N$, then 
  \begin{eqnarray*}
    & \label{IINVC9} 1\leq \widehat{D}_N({k})\leq N+1   & \forall\, {k} \in \Z^3,
    \\
    & \label{JJV34}  \displaystyle \widehat D_N ({k}) \approx (N+1)
    \frac{1+\alpha^2 |{k}|^2}{\alpha^2 |{k}|^2 } & \text{ for
      large } |{k}|,  
    \\
    &  \label{TT0ON}  \displaystyle  \lim_{|{k}
      |\to+\infty}\widehat{D}_N({k})=N+1  & \text{for fixed    }\alpha>0,  
    \\  
    & 
    \widehat{D}_N({k})\leq(1+\alpha^2|{k}|^2)  &  \forall\,
    {k} \in \Z^3
    ,\     \alpha>0.   
  \end{eqnarray*} 
\end{lemma}
This lemma shows that in the specific case of practical use, we can
assume that the upper and lower bounds for $\widehat{D}_N$ are $d_0=1$
and $d_1=N+1$, which reflects in similar bounds for $D_{N}$.
\subsection{Known theory for RADM solutions}

We define now the notion of weak solution for the RADM.
\begin{definition}
  We say that $v\in L^\infty(0,T;V^1)$ is a weak solution to~\eqref{eq:cadm1}-\eqref{eq:cadm2}
  if 
  \begin{equation*}
    \frac{d}{dt}   \left[( v,\phi)+  \alpha^2(\nabla v,\nabla\phi)\right]    + (D v \cdot \nabla D v,\phi)
    + \nu (\nabla D v,\nabla \phi)
    =  (f,\phi),
  \end{equation*}
  for all $\phi\in V^1$, and also provided that $v_t\in
  L^2(0,T;V^{-1})$ and the following energy identity holds true:
  \begin{multline}
    \label{radme1}
    \frac{1}{2}\left( \alpha^2\norm{\nabla D^{1/2} v(T)}^2 +
      \norm{D^{1/2} v(T)}^2\right) + \nu\int_0^T \norm{\nabla Dv}^2\,dt
    \\
    = \frac{1}{2}\left( \alpha^2\norm{\nabla D^{1/2} v(0)}^2 +
      \norm{D^{1/2} v(0)}^2\right) + \int_0^T (f,Dv)\,dt.
  \end{multline}
\end{definition}
We have the following existence and uniqueness theorem, which is
based on a standard Galerkin approach, together with the energy
estimate~\eqref{radme1} below, see~\cite[Thm.~2.1]{R13b}.
\begin{theorem}
\label{thm:existence} 
  Let $f\in L^2(0,T;H^{-1}(\Omega))$ and $v_0\in V^1$ be given.  Then,
  there exists a unique weak solution to~\eqref{eq:cadm1}-\eqref{eq:cadm2}.
\end{theorem}

\medskip

The proof of this result is standard and follows the same guidelines
of the results in~\cite{BB12,CLT2006}, with the key point being the
energy identity~\eqref{radme1}.  From~\cite{R13b}, with the
assumptions on the data, we have that the solution $v\in
L^{\infty}(0,T;V^1)$, and so all manipulations in this proof are
justified. Multiplying~\eqref{eq:cadm1} by $Dv$
and integrating over the physical domain gives
  \begin{equation*}
    \alpha^2(\nabla v_t,\nabla Dv) + (v_t,Dv) + (Dv\cdot\nabla Dv,Dv) -
    (q,\nabla \cdot Dv) + \nu\norm{\nabla Dv}^2 = (f,Dv),
  \end{equation*}
  and then using a vector identity on the first two terms, and that
  $\nabla \cdot Dv=0$, both the pressure and nonlinear terms vanish,
  leaving us with
  \begin{equation*}
    \frac12\frac{d}{dt} \left( \alpha^2\norm{\nabla D^{1/2}v}^2 +
      \norm{D^{1/2}v}^2 \right) + \nu\norm{\nabla Dv}^2 = (f,Dv).
  \end{equation*}
  Integrating in time over $[0,T]$ provides~\eqref{radme1}.

  This equality shows that the Galerkin approximate solution
  $v^{m}=\sum_{1\leq|k|\leq m}\widehat{v}_{k} e^{i{k \cdot x}}$ is
  bounded uniformly in $L^\infty(0,T;V^1)$. Then, by making use of the
  comparison argument one can easily prove that
  \begin{equation*}
    v_{t}^{m}-\alpha^2\Delta v_{t}^{m} \in L^2(0,T;V^{-1}).
  \end{equation*}
  The Lax-Milgram lemma set in the space $V^1$ implies that $v^m_t\in
  L^2(0,T;V^1)$. In particular, this proves that one can extract from
  the Galerkin approximate functions a converging sub-sequence, and
  moreover shows also that the solution $v\in C([0,T];V^1)$, which is
  a uniqueness class.  The regularity of the time derivative
  allows to show that an energy equality (instead of an inequality as
  for the NSE) holds true.
  
  \bigskip
  
  In the case of the problem without viscosity, i.e. the RADM-Euler
  equations 
  \begin{eqnarray}
  v_t -\alpha^2\Delta v_t + D v \cdot \nabla D v+ \nabla q
  & = & f, \label{Ecadm1} 
  \\
  \nabla \cdot Dv & = & 0, \label{Ecadm2} 
\end{eqnarray}
by adapting techniques from~\cite{BB12,CLT2006} and by using the
estimates on the operator $D$ from Lemma~\ref{lem:lower_bound}, it is easy
to infer the following result (the same applies also to the
Euler-Voigt-deconvolution model, as discussed in the next section)
\begin{theorem} 
  \label{thm:existence-Titi}
  Let $T > 0,\,m \geq 1$, and $v_0\in V^m$. Let $f \in
  C([-T,T];W^{m-1,6/5})$ with $\nabla\cdot f=0$. Then, there exists a
  unique solution $v$ of the RADM Euler
  equations~\eqref{Ecadm1}-\eqref{Ecadm2} which belongs to
  $C^1([-T,T];V^m)$.  Moreover,
  \begin{equation*}
    \sup_{t\in[-T,T]}\|v (t) \|_{V^m} < C(\alpha,
    \|v_0\|_{V^m},\sup_{-T<t<T}\|f(t)\|_{W^{m-1,6/5}},T,d_0,d_1).
  \end{equation*}
\end{theorem}
\section{Remarks on the derivation of the model and a connection to
  NS-Voigt}
\label{sec:derivation}
Interestingly, despite its derivation as an altered/reduced ADM, the
proof for well-posedness of the system in~\cite{R13b} uses ideas of
the well-posedness proof for NS-Voigt from~\cite{BB12}.  This suggests
a connection between the two models, and we show now that a NS-Voigt
deconvolution model can be derived from~\eqref{eq:cadm1}-\eqref{eq:cadm2}.
Writing
\begin{equation*}
v_t - \alpha^2\Delta v_t = F^{-1}v_t =F^{-1}D^{-1}Dv_t=D^{-1}F^{-1}(Dv)_t,
\end{equation*}
and setting $w=Dv$ yields
\begin{eqnarray}
  D^{-1}F^{-1}w_t  + w \cdot \nabla w + \nabla q - \nu \Delta w
  & = & f, \label{v1}
 \\
  \nabla \cdot w & = & 0. \label{v2} 
\end{eqnarray}
This system~\eqref{v1}-\eqref{v2} is precisely the NS-Voigt
regularization except for the $D^{-1}$ term included in the
regularization, and thus we could refer to~\eqref{v1}-\eqref{v2} as
the NS-Voigt-deconvolution regularization.  Since the usual
regularization in NS-Voigt is given by just $F^{-1}$, and since
$D\approx F^{-1}$, this ``new'' regularization $D^{-1}F^{-1}$ will
provide a higher order of consistency for~\eqref{v1}-\eqref{v2} to the
NSE (which are recovered if no regularization is applied).

In this section we enlarge the derivation of the Voigt models as given
in~\cite{R13b}, and in particular we highlight the fact that Voigt
models can be seen as ADMs. The Euler-Voigt equations read as
\begin{eqnarray}
  v_t   - \alpha^2\Delta v_t+v \cdot \nabla  v + \nabla q
  & = & f, \label{Eadm1} 
  \\
  \nabla \cdot Dv & = & 0, \label{Eadm2} 
\end{eqnarray}
and have been introduced in the LES community by Cao, Lunasin, and
Titi~\cite{CLT2006}. They started from the inviscid ``simplified
Bardina'', which is the same model studied in Layton and
Lewandowski~\cite{LL06} (zeroth order Stolz and Adams) in the case
$\nu=0$. The exact words are ``\textit{...Therefore, we propose the
  inviscid simplified Bardina model as regularization of the 3D Euler
  equations that could be implemented in numerical computations of
  three dimensional inviscid flows. }'' They then re-introduced the
viscosity, obtaining the NS-Voigt
equations~\eqref{eq:cadm1}-\eqref{eq:cadm2}.  Later, in Larios and
Titi~\cite{LT10}, the model is justified as
follows ``\textit{\dots it was noted in Cao Lunasin and
  Titi~\cite{CLT2006} that formally setting $\nu=0$ in simplified
  Bardina amounts to simply adding the term $-\alpha^2\Delta u_t$ to
  the left-hand side of the Euler equations yielding what they call
  the Euler-Voigt equations. Remarkably, if one reintroduces the
  viscous term $-\nu\Delta u$ the resulting equations happen to
  coincide with equations governing certain visco-elastic fluids known
  as Kelvin-Voigt fluids, which were first introduced and studied by
  A.P.  Oskolkov~\cite{Osk1973}.}''

We show now that the NS-Voigt model can also be derived with usual
modeling techniques and in the unified framework of deconvolution
models.  In addition to the motivation of the order reduction
explained in~\cite{GRT13}, we follow now a pure approximate
deconvolution derivation.  

In order to the derive the model, apply the filter
$F=:A^{-1}=(I-\alpha^2\Delta)^{-1}$ to the incompressible
NSE~\eqref{eq:NSE} to get
\begin{equation*}
  A^{-1}{u}_t+A^{-1}\big[{(u\cdot\nabla)\,u}\big]-\nu A^{-1}\Delta u+\nabla
  A^{-1} p=A^{-1}f.
\end{equation*}
Rewriting the equation only in terms of the filtered variable
$(\overline{u},\overline{p})$ we get
\begin{equation*}
  \overline{u}_t+A^{-1}\big[{(A\overline{u}\cdot\nabla)\,A\overline{u}}\big]-\nu
  A^{-1}\Delta A\overline{u}+\nabla \overline{p}=A^{-1}f,
\end{equation*}
which is the starting point in the Adams and Stolz~\cite{AS99,AS01} approach. 

Introduce now a family of abstract deconvolution operators $\{D_N\}_N$
which are of order zero (explicit examples are those recalled in
Section~\ref{sec:app-dec}, see especially properties from
Lemma~\ref{lem:lower_bound}) such that
\begin{equation*}
  D_0=I\quad\text{and }\quad D_N\to A \text{ as }N\to+\infty.
\end{equation*}
We then insert the deconvolution into the equations, which allows us to write 
\begin{equation*}
  A\overline{u}\approx D_N\overline{u}.
\end{equation*}
Taking the zeroth order approximation, we get (in the new
approximate variables $v\sim \overline{u}=D_0\overline{u}$ and $r\sim
\overline{p}$),
\begin{equation*}
  {v}_t+A^{-1}\big[{({v}\cdot\nabla)\,{v}}\big]-\nu
  A^{-1}\Delta v+\nabla r=A^{-1}f.
\end{equation*}
Finally, by applying the operator $A$ to both sides, we arrive at
the NS-Voigt equations
\begin{eqnarray} 
  \label{eq:NS-Voigt}
  {v}_t+\alpha^2\Delta v_t+{(v\cdot\nabla)\,v}-\nu \Delta
  v+\nabla q&=&f,
  \\
  \nabla \cdot v & = & 0, \label{eq:NS-Voigt2} 
\end{eqnarray}
Following the same procedure, but using a higher accuracy operator
$D_N$ (as in~\eqref{v1}) we obtain the NS-Voigt deconvolution model
\begin{equation*}
  {v}_t+A^{-1}\big[{(D_N{v}\cdot\nabla)\,D_N{v}}\big]-\nu
  A^{-1}\Delta D_N v+\nabla q=0.
\end{equation*}
After applying $A$, we recover exactly the RADM model~\eqref{eq:cadm1}
studied in~\cite{GRT13,R13b}.
\section{Energy analysis}
\label{sec:invariants}
This section studies the treatment of energy of the RADM system.  We
begin with the energy balance, which is useful both in obtaining a
priori estimates for existence theorems and as a starting point to
study the RADM energy cascade.  Then, a Kraichnan-type analysis of the
RADM energy cascade is given, followed by some numerical experiments
to test the energy cascade in forced, homogeneous, and isotropic turbulent
flows.
\subsection{Energy balance}
Energy is a quantity of fundamental importance in fluid flow modeling,
both in terms of the physical relevance of its solutions and its
potential for success in simulations.  For a model to be successful in
accurate coarse mesh simulations, it must dissipate sufficient energy
to remain stable.  We show here that the deconvolution in the RADM
acts as an energy sponge, which seems to explain a smoothing effect
from increasing deconvolution in the RADM that was observed
in~\cite{GRT13} with van Cittert approximate deconvolution, and
in~\cite{G13} with multiscale deconvolution.  

We present now some integral identities and estimates, which -at the
level of Faedo-Galerkin approximations- are the core to obtain
existence of weak solutions. At the level of weak solutions (which are
smooth enough) they are still valid and show the balance of the main
flow quantities.
%
\begin{remark}
  Equation~\eqref{radme1} shows that the RADM conserves a model
  energy 
  \begin{equation}
    E_{M}(t):=\frac{1}{2}\left( \alpha^2\norm{\nabla D^{1/2} v(t)}^2 +
      \norm{D^{1/2} v(t)}^2\right). \label{eq:modelenergy}
  \end{equation}
  The conservation of the energy is one of the outstanding open
  problems for the Navier-Stokes equations. The analysis of the
  conservation of the model energy for LES models started with the
  analysis in~\cite{Ber2012b,BL2012}, and several ADM models share
  this property.
\end{remark}
\begin{remark}
  We note that since 
  $ \| D \|\ge 1$, the balance~\eqref{radme1} suggests that
  deconvolution in the RADM increases the effective viscosity.  This
  is consistent with numerical results in~\cite{G13,GRT13}, which
  reported a smoothing effect by increasing the deconvolution order.
\end{remark}
\subsection{Energy spectra}
This section presents an analytic and computational study of the RADM
energy spectra.  A widely accepted theory of turbulent flows starting
with intuitions of L.F.~Richardson and A.N.~Kolmogorov is that
energy is input at large scales, transferred from large to small
scales through the inertial range, and finally dissipated at small
scales by viscosity~\cite{F95,FMRT2001a,BIL06}.  It is further
believed that viscosity has no effect except on very small scales,
on which it acts to decay exponentially fast.

To fully resolve a turbulent flow, one must resolve the NSE on this
entire range of active scales, which is far beyond the limit of modern
computational tools for high Reynolds number problems.  Computing the
NSE without full resolution is widely known to be unstable.  For a
model to be computable in practice, its range of active scales needs
to be significantly smaller than for the NSE, and thus studying its
energy spectra can help make such a determination.

Since the RADM takes a form similar (in some sense) to that of
Leray-$\alpha$, NS-$\alpha$, and the simplified Bardina models, we
adapt physics-based energy spectra analysis performed
in~\cite{CLT2006,CHOT05,FHT01} for these respective models.  We begin
with a decomposition of the velocity into its Fourier modes, which
gives the energy balance
\begin{equation}
  \label{eq:spec1}
  \frac12 \frac{d}{dt}\left( (v_k,\widehat{D}_N(k) v_k) +
    \alpha^2(-\Delta v_k,\widehat{D}_N(k) v_k) \right)
  + \nu(-\Delta \widehat{D}_N(k) v_k,\widehat{D}_N(k) v_k) = T_k - T_{2k}, 
\end{equation}
where $\widehat{D}_N(k)$ is the Fourier transform of the deconvolution
operator $D_N$, and
\begin{equation*}
  T_k := -(( \widehat{D}_N(k) v_k^< \cdot\nabla \widehat{D}_N(k) v_k,
  \widehat{D}_N(k) v_k) + (( \widehat{D}_N(k) (v_k + v_k^>) \cdot\nabla
  \widehat{D}_N(k) (v_k + v_k^>), \widehat{D}_N(k) v_k^<), 
\end{equation*}
with
\begin{equation*}
  v_k^< := \sum_{|j|<k} v_j,\quad \text{and} \quad v_k^> := \sum_{|j|\ge
    2k} v_j. 
\end{equation*}
Note that $T_k - T_{2k}$ represents the net amount of energy
transferred into wave numbers between $[k,2k)$.  Time averaging the
energy balance~\eqref{eq:spec1} gives the energy transfer equation
\begin{equation}
  \label{eq:spec2} 
  \langle \nu(-\Delta \widehat{D}_N(k) v_k,\widehat{D}_N(k) v_k)\rangle
  = \langle T_k \rangle - \langle T_{2k} \rangle. 
\end{equation}
Since the RADM model energy is defined by~\eqref{eq:modelenergy},
as this is the conserved quantity in the absence of viscosity and
external forces, we take $D=D_N$ and define the model energy of an
eddy of size
$k^{-1}$ by
\begin{equation*}
  E_M(k) := \widehat{D}_N(k) (1+\alpha^2 k^2) \sum_{|j|=k} | \hat{v}_j |^2.
\end{equation*}
Combining this with the energy transfer equation~\eqref{eq:spec2} gives the relation
\begin{equation*}
  \langle T_k \rangle - \langle T_{2k} \rangle
  \sim
 \frac{ \nu k^3 \widehat{D}_N(k) E_M(k)}{1+\alpha^2 k^2}
 \sim
 \frac{ \nu k^3 (N+1) E_M(k)}{1+\alpha^2 k^2},
\end{equation*}
with the last relation coming from Lemma~\ref{lem:lower_bound}, since
$\widehat{D}_N(k)\sim (N+1)$ for sufficiently large $k$.  If the wave
number $k$ belongs to the inertial range, then $ \langle T_k \rangle
\approx \langle T_{2k} \rangle$ and there is no leakage of energy
through dissipation.  Then, we have that $k$ is in the inertial range
if
\begin{equation*}
  \frac{ \nu k^3 (N+1) E_M(k)}{1+\alpha^2 k^2}
  \ll
  \langle T_k \rangle.
\end{equation*}
This shows that increasing the deconvolution order $N$ has the effect
of shortening the inertial range of the RADM.  This is consistent with
the energy balance given above, which suggests that increasing $N$
reduces the total energy and increases the effective viscosity.

To determine the kinetic energy distribution in the inertial range, we
begin by defining the average velocity of a size $k^{-1}$ eddy, and
relating it to model energy via
\begin{equation*}
  U_k := \langle (v_k,v_k) \rangle ^{1/2} \sim \left( \int_k^{2k}
    \frac{E_M(k)}{(1+\alpha^2 k^2) \widehat{D}_N(k)} \right)^{1/2}
  \sim \left( \frac{ k E_M(k) }{(1+\alpha^2 k^2) (N+1) }
  \right)^{1/2}. 
\end{equation*}
The total energy dissipation rate is determined by the energy
balance to be
\begin{equation*}
\epsilon_{M} := \langle \nu \| \nabla D v \|^2 \rangle,
\end{equation*}
and from the Kraichnan energy cascade theory gives that the
corresponding turnover time for eddies of this size is
\begin{equation*}
\tau_k := \frac{1}{k U_k} = \frac{ (1+k^2\alpha^2)^{1/2} (N+1)^{1/2}  }{k^{3/2}E_M(k)^{1/2}},
\end{equation*}
the model energy dissipation rate $\epsilon_M$ scales like
\begin{equation*}
\epsilon_M \sim \frac{1}{\tau_k}\int_k^{2k} E_M(k) \sim
\frac{k^{5/2}E_M(k)^{3/2} }{ (1+k^2\alpha^2)^{1/2} (N+1)^{1/2}  }. 
\end{equation*}
Solving for $E_M(k)$ provides the relation
\begin{equation*}
E_M(k) \sim \frac{\epsilon_M^{2/3} (1+\alpha^2 k^2)^{1/3}(N+1)^{1/3}   }{k^{5/3}   },
\end{equation*}
which implies a kinetic energy ($E = \frac{1}{2} \norm{v}^2$) spectrum
\begin{equation*}
  E(k) \sim \frac{ E_M(k) }{(1+\alpha^2 k^2) \widehat{D}_N(k) } \sim
  \frac{\epsilon_M^{2/3}    }{k^{5/3}(1+\alpha^2 k^2)^{2/3}
    (N+1)^{2/3}  }. 
\end{equation*}
As is the case for related models, such as those of $\alpha$-type or
the simplified Bardina model, the kinetic energy spectrum can be
divided into 2 parts.  If $k\,\alpha > O(1)$, then
\begin{equation*}
  E(k) \sim \frac{\epsilon_M^{2/3}    }{k^{5/3}(\alpha^2 k^2)^{2/3}
    (N+1)^{2/3}  } \sim \frac{\epsilon_M^{2/3}    }{k^{3} \alpha^2
    (N+1)^{2/3}  }, 
\end{equation*}
and if $k\,\alpha < O(1)$, then we get
\begin{equation*}
E(k) \sim \frac{\epsilon_M^{2/3}    }{k^{5/3} (N+1)^{2/3}  }.
\end{equation*}
These scalings suggest that for larger scales in the inertial range, a
$k^{-5/3}$ roll-off of energy is expected, but for smaller wave
numbers, the slope increases to $k^{-3}$.  This implies that
significantly less energy will be held in higher wave numbers,
suggesting the model is indeed more easily computable than the NSE,
which has a $k^{-5/3}$ roll-off of energy through its entire inertial
range.  
\subsection{Numerical testing of energy spectra}
To test some of these scaling results, a direct numerical simulation
(DNS) of the non-dimension\-alized RADM is performed for forced
turbulence, which is one of the most extensively simulated turbulent
flows (see e.g.,~\cite{Bowers13,CDKS93,TYKIM09,TYKIM12,TYKIM14}). In
particular, we study the energy spectrum of the model for
three-dimensional homogeneous and isotropic turbulent flow, on a
periodic cubic box with a side length $2\pi$, with Reynolds number $Re=200$. 

We use a pseudo-spectral
method for the spatial discretization of the model with full
de-aliasing, and a second-order Adams--Bashforth scheme for the time
stepping. In addition, we employ a forcing method used by Chen et
al.~\cite{CDKS93,CHMZ99} and She et al.~\cite{SJO1991}.  The
numerical forcing of a turbulent flow is the artificial addition of
energy at the large scales in a numerical simulation. Statistical
equilibrium is achieved by the balance between the input of kinetic
energy through the forcing and its output through the viscous
dissipation. 
Van Cittert deconvolution is used with order $N=0,1,2$, and we will
specify $N$ in the plots.

We first begin by testing the scaling law in the previous section,
which predicts that the kinetic energy scales with $k^{-5/3}$ at the
beginning of the inertial range, and transitions to $k^{-3}$ near the
end of the inertial range.  A plot of the resulting energy spectra
from a computation with $N=2$ and $\alpha=1/32$ at the resolution of
$128^3$ is shown in Figure~\ref{radmscaling}.  Along with the
spectrum, on the log-log plot are also lines with slopes -5/3 and -3,
and we observe that the spectrum appears to be in agreement with these
slopes near the beginning and end of the inertial range, respectively.
\begin{figure}[!ht]
\begin{center}
\includegraphics[width=.55\textwidth,height=0.5\textwidth, viewport=0 0
480 410, clip]{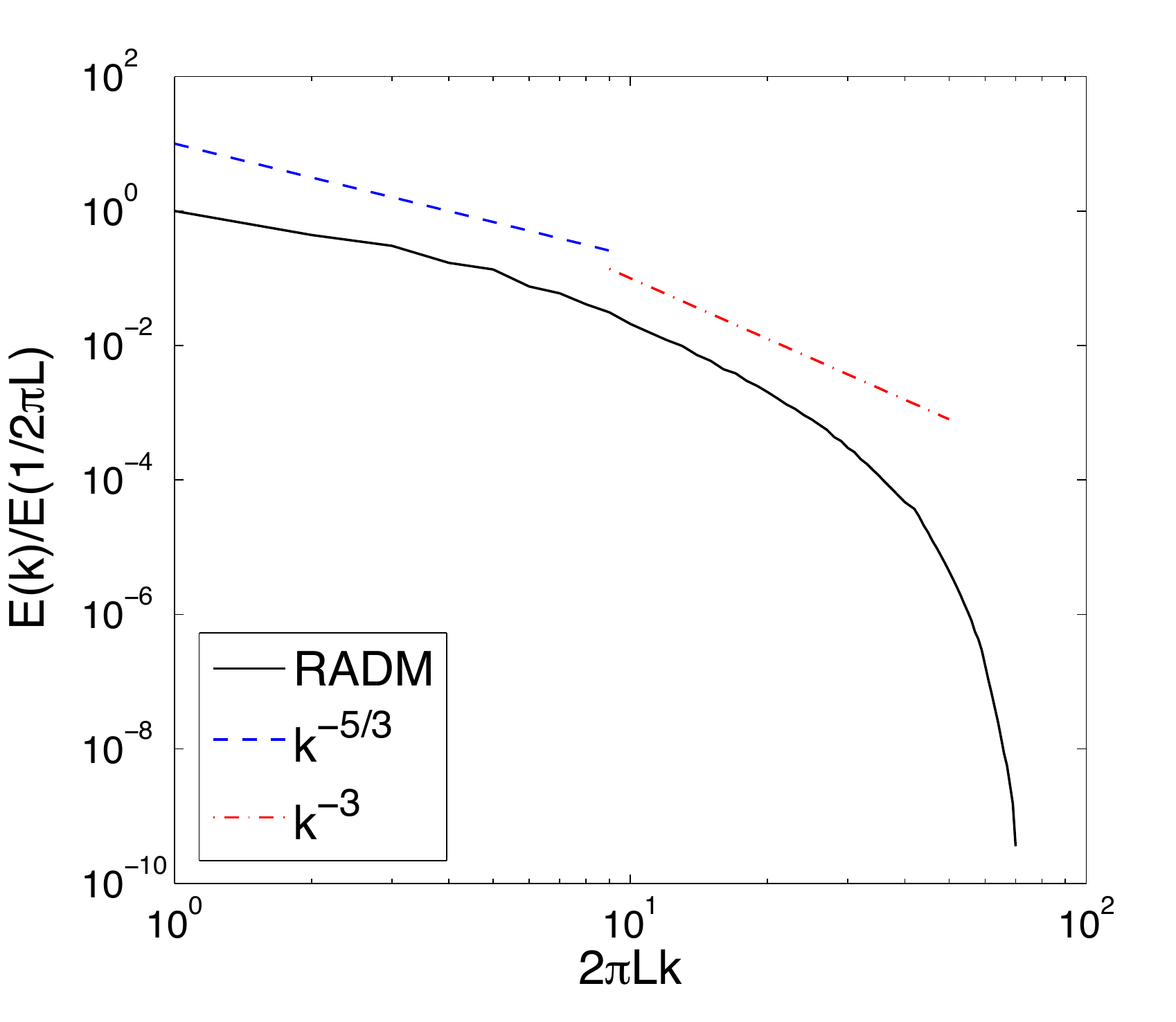}
\caption{ \label{radmscaling}  The energy spectrum of the RADM at
  $128^3$ resolution ($N=2$, $\alpha=1/32$).  Slopes on the $\log$-$\log$
  plot appear in agreement with the predicted slopes of $-5/3$ and $-3$ at
  the beginning and end of the inertial range, respectively.} 
\end{center}
\end{figure}


We also investigate the sensitivity of $N$ on the energy spectra of
the RADM.  Here, we compute with $128^3$ resolution, $\alpha=1/16$ and
$N=0,1,2$.  Results are shown in Figure~\ref{espchN}, and we observe
(as our analysis predicts) the RADM energy curves shorten as $N$
increases; that is, as $N$ is increased, there is less (or equal)
energy in each wave number across the entire spectrum.  This shows
there is less total energy in the $N=2$ system than $N=1$, and $N=1$
compared to $N=0$, which is in agreement with the above analysis.
\begin{figure}[!htbp]
\begin{center}
\includegraphics[width=.55\textwidth,height=0.5\textwidth, viewport=0 0
480 440, clip]{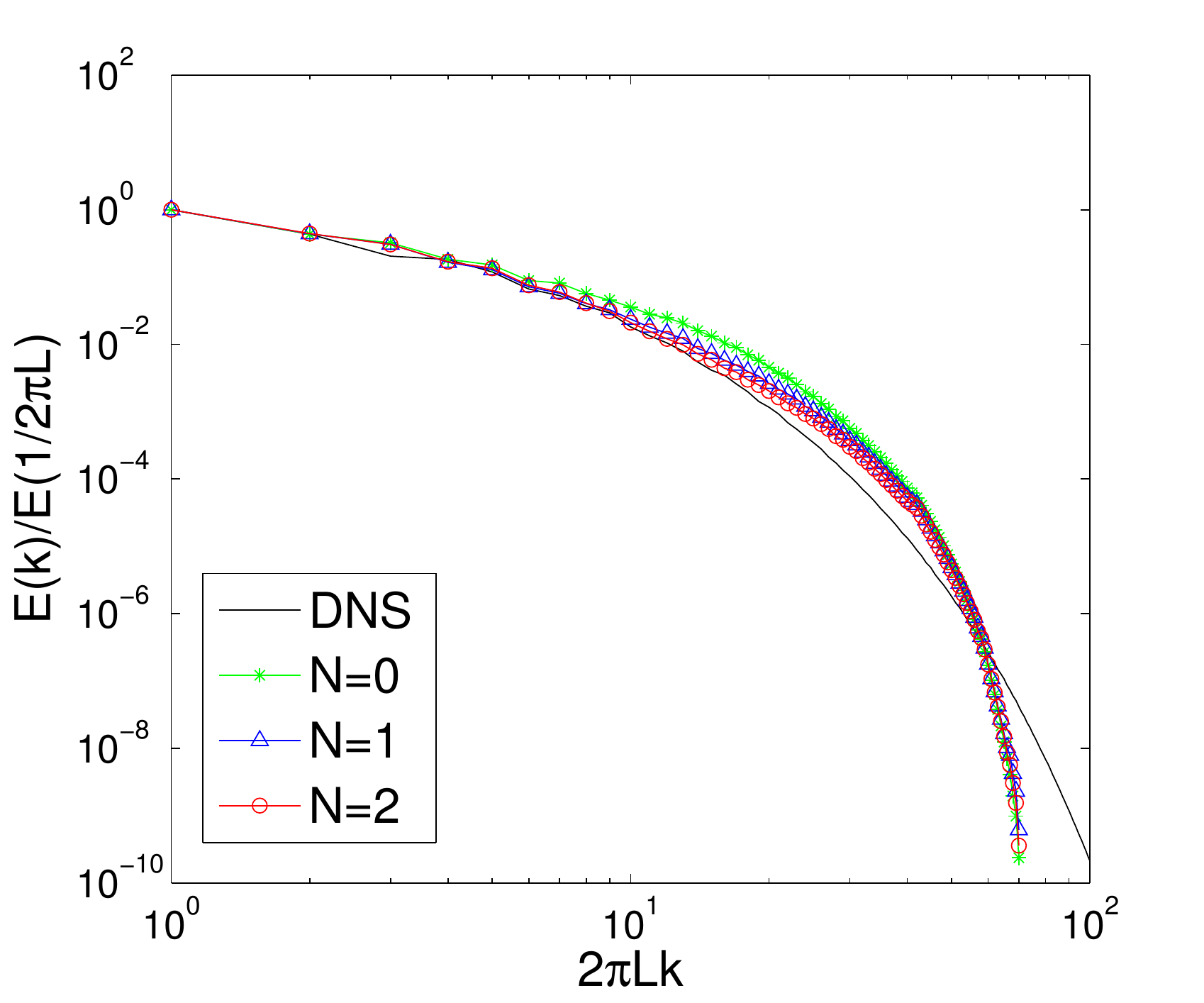}
%
 \end{center}
 \caption{Energy spectra for values of $N=0$, $1$, $2$ at
   $128^3$ resolution for $\alpha=1/16$, along
   with the energy spectrum obtained from the DNS of the NSE at
   $256^3$ resolution. The proposed model's energy spectra become
   closer to the DNS energy spectrum with increasing $N$.}
 \label{espchN}
\end{figure}


Lastly, we use this experiment setup to test the accuracy of the RADM
energy spectra against a DNS of the NSE.  In Figure~\ref{espchaADM},
we show the resulting energy spectra for the RADM with various values
of $\alpha$ for $N=2$ at the resolutions of $63^3$ and $128^3$ along
with the DNS energy spectrum obtained from the NSE at $256^3$
resolution. The energy spectra are provided for three cases of
$\alpha=1/8$, $1/16$, and $1/32$. As $\alpha$ decreases, the energy
spectra of the RADM become closer to the highly resolved DNS energy
spectrum.

\begin{figure}[!htbp]
\begin{center}
  \begin{picture}(500,220)
    \put(0,5){\epsfig{file=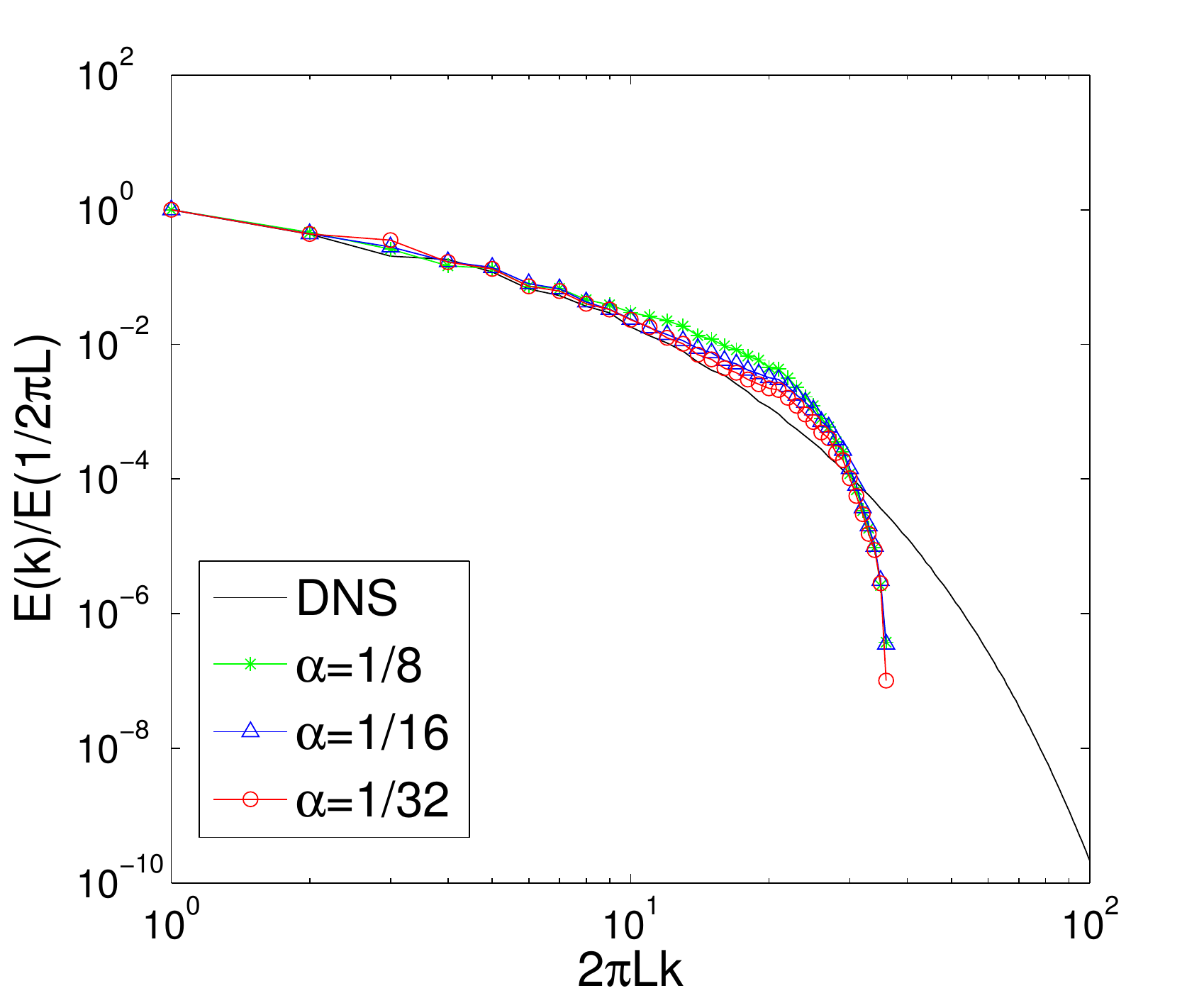,width=8.0cm}}
    \put(220,5){\epsfig{file=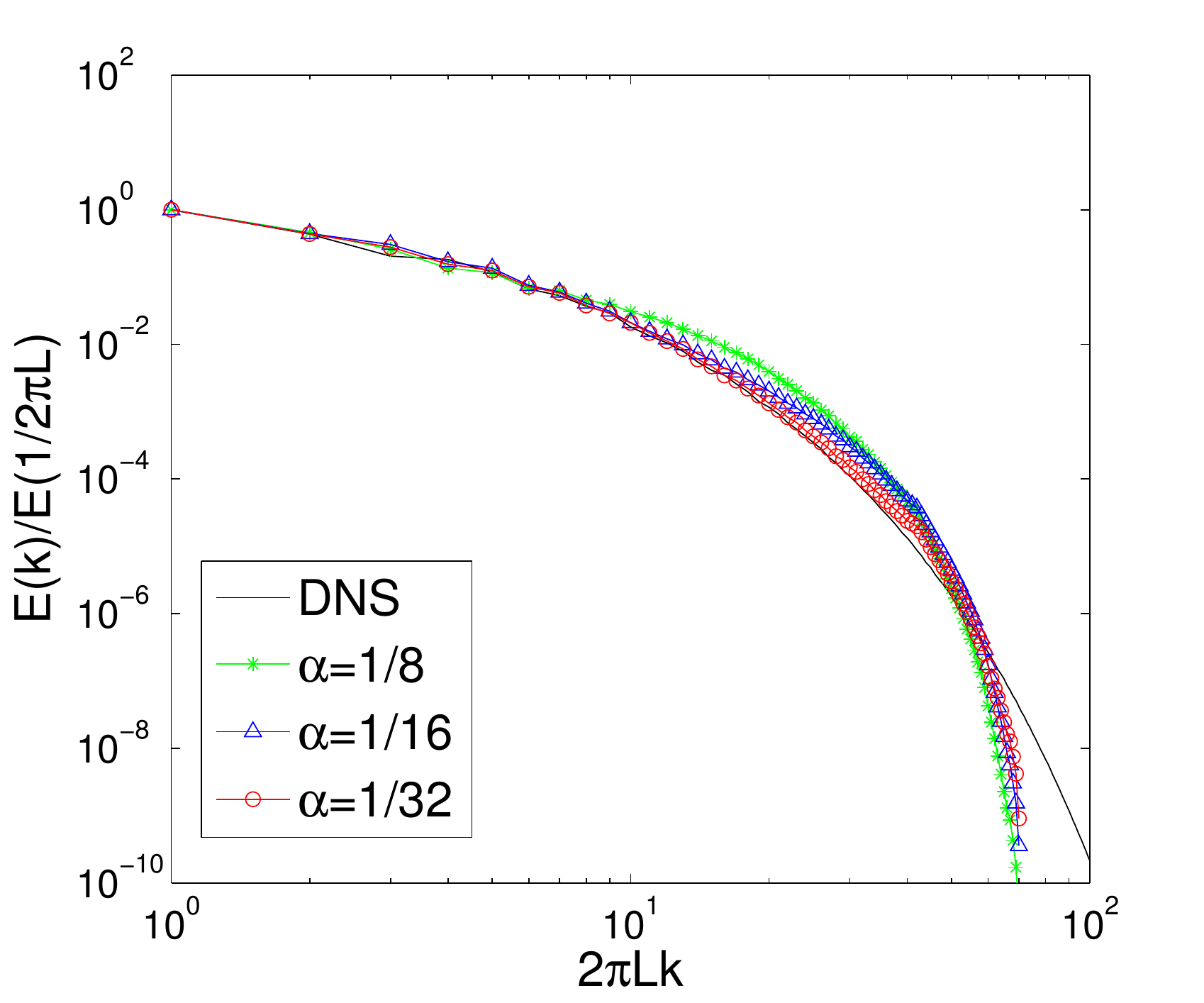,width=8.0cm}}
    \put(110,-5){(a)} \put(330,-5){(b)}
\end{picture}
\end{center}
 \caption{Energy spectra for various values of $\alpha=1/8$,
   $1/16$, and $1/32$ for $N=2$ at (a) $64^3$ and (b) $128^3$
   resolutions along with the energy spectrum obtained from the DNS of
   the NSE at $256^3$ resolution. The proposed model's energy spectra
   become closer to the DNS energy spectrum with decreasing $\alpha$.} 
 \label{espchaADM}
\end{figure}
\section{Existence of global attractors for the RADM}
\label{sec:attractors}
The long-time behavior of evolution equations can be characterized by
their attractors, which in many important cases are finite
dimensional. In particular, the notion of global attractor
(cf.~\cite{T97}) seems the most relevant in the context of the NSE. It
is well known that such an attractor exists for the 2D NSE, while the
lack of well-posedness in the 3D case is reflected in partial results
concerning weak or trajectory attractors. In the case of LES models,
the solution has naturally better regularity properties (even if a
certain hyperbolic behavior is expected), hence it is natural to
expect existence of such an object. In particular, for the NS-Voigt
model the analysis of these questions has been addressed
in~\cite{KT09,KLT2009}, where existence and regularity of the global
attractor and also determining modes are studied. Since the NS-Voigt
model is a peculiar system (pseudo-parabolic, see also results in
Section~\ref{sec:pulsatile}) with also non-viscous features, the proof
of existence of the global attractor uses techniques typical of
hyperbolic systems and the semigroup results to be only asymptotically
compact. In particular see the questions related with the regularity
of the solution in the space variables recalled in
Sec~\ref{sec:regularity}.

In this section we follow very closely the approach in~\cite{KT09} and
we just highlight the changes needed to adapt the proof to the new
RADM system. The main difference with respect to the latter reference
is that, due to the presence of approximate deconvolution operators,
we need to deal with the space-periodic case.

For simplicity of analysis and notation, following~\cite{KT09} in this
section we consider the model~\eqref{eq:cadm1}-\eqref{eq:cadm2} in the
equivalent form
\begin{eqnarray}
  v_t  + \alpha^2 A v_t + P(D v \cdot \nabla D v) + \nu A D v =
  Pf, \label{eq:radm1} 
\end{eqnarray}
where $P$ is the Helmholtz-Leray orthogonal projection in $L^2(\tore)$
onto $V_0$, and $A:=-P\Delta$ is the Stokes operator with periodic
boundary conditions. The existence of a global attractor will follow
from a well-known result, for which we refer for instance
to~\cite{T97}.
\begin{theorem}
  \label{thm:attract}
  Assume a semigroup $S(t):V^1\rightarrow V^1,\ 0< t_0\le t$ can be
  decomposed as
\begin{equation*}
S(t) = Y(t) + Z(t),
\end{equation*}
where $Z(t)$ is compact in $V^1$ for each $0< t_0\le t$.  Assume also
there exists a continuous $k:[t_0,\infty) \times \Real^+
\rightarrow\Real^+$ such that for every $r>0,\ k(t,r)\rightarrow 0$ as
$t \rightarrow \infty$ and that
\begin{equation*}
  \norm{Y(t)v}_{V} \le k(t,r),\quad \forall\, t\ge t_0,\ \forall\, v\in
  V^1:\ \norm{v}\le r.
\end{equation*}
Then $S(t):V^1\rightarrow V^1,\ t\ge 0$, is asymptotically compact.
\end{theorem}
\medskip

The main result of this section is the following.
\begin{theorem}
  Let $v_0\in V^1$ and assume the forcing is only spatially dependent,
  $Pf=Pf(x)\in V^{-1}$.  Let $S(t):V^1 \rightarrow V^1,\ t\in
  \Real^{+}$ be the continuous semigroup generated by the
  problem~\eqref{eq:radm1} (the existence of which is guaranteed by
  Theorem~2.1 of~\cite{R13b}).  Then $S(t)$ has a global attractor.
\end{theorem}
\begin{proof}
  As usual we begin by showing the existence of an absorbing ball.
  Then, we show that the assumptions of Theorem~\ref{thm:attract} are
  satisfied, to conclude that a global attractor exists.  We proceed
  formally, but calculations again can be completely justified by a
  Faedo-Galerkin approximation.

  We multiply~\eqref{eq:radm1} by $Dv$ and proceed as above for the
  energy inequality analysis, along with Cauchy-Schwarz, and Young
  inequalities on the forcing term to obtain the bound
\begin{equation*}
  \frac{d}{dt} \left( \alpha^2\norm{\nabla D^{1/2}v}^2 + \norm{D^{1/2}v}^2 \right)
  + \nu\norm{\nabla Dv}^2 \le \frac{1}{\nu} \norm{Pf}_{-1}^2.
\end{equation*}
Using that $D$ is positive and self-adjoint, along with assumptions on its boundedness, gives
\begin{equation*}
\frac{d}{dt} \left( \alpha^2\norm{\nabla D^{1/2}v}^2 + \norm{D^{1/2}v}^2 \right)
 + \nu \,d_0\, \norm{\nabla D^{1/2} v}^2 \le\frac{1}{\nu} \norm{Pf}_{-1}^2,
\end{equation*}
and then applying Poincar\'e provides the bound
\begin{equation*}
\frac{d}{dt} \left( \alpha^2\norm{\nabla D^{1/2}v}^2 + \norm{D^{1/2}v}^2 \right)
 + {\nu d_0}\left( \frac{\alpha^{2}}{\alpha^2} \norm{\nabla D^{1/2}
     v}^2 + \lambda \norm{ D^{1/2} v}^2 \right)  
  \le \frac{1}{\nu} \norm{Pf}_{-1}^2.
\end{equation*}
Setting $c_0= \frac{\nu d_0}{2} \min \{\alpha^{-2},\lambda \}$, we get
\begin{equation*}
\frac{d}{dt} \left( \alpha^2\norm{\nabla D^{1/2}v}^2 + \norm{D^{1/2}v}^2 \right)
 + c_0 \left( \alpha^2 \norm{\nabla D^{1/2} v}^2 +  \norm{ D^{1/2} v}^2 \right) 
  \le \frac{1}{\nu} \norm{Pf}_{-1}^2,
\end{equation*}
and thus from the Gronwall inequality,
\begin{equation*}
\limsup_{t\rightarrow +\infty} \left( \alpha^2\norm{\nabla D^{1/2}
    v(t)}^2 + \norm{D^{1/2} v(t)}^2 \right)
\le 
 \frac{\norm{Pf}_{-1}^2}{\nu c_0}.
\end{equation*}
Again applying the lower boundedness assumption on $D$, this reduces to
%
\begin{equation*}
  \limsup_{t\rightarrow \infty} \left( \alpha^2\norm{\nabla v(t)}^2 + \norm{v(t)}^2 \right)
  \le 
  \frac{2\max\{\alpha^2,\lambda^{-1}\}} {\nu^2\,d_0^2}\norm{Pf}_{-1}^2:=M_1^2,
\end{equation*}
which implies that $S(t):V^1\rightarrow V^1,\ t\in \Real^+$ has as
absorbing ball $B_1:=B(0,M_1)\subset V^1$,
%
and we have the uniform estimate
\begin{equation*}
\norm{v(t)}_1\le M_1,
\end{equation*}
for $t$ large enough.

We now show the assumptions of Theorem~\ref{thm:attract} are satisfied,
following analysis similar to~\cite{KT09}.  Observe that $S(t)$ can be
represented as
\begin{equation*}
S(t)v_0 = Y(t)v_0 + Z(t)v_0,
\end{equation*}
where $Y(t)$ is the semigroup generated by the linear (which is linear
since also $D$ is a linear operator when the van Cittert deconvolution
is used) problem
\begin{eqnarray}
  y_t + \alpha^2A y_t+ \nu ADy  = 0, \label{eq:linear1}
  \\ 
  y(0)=v_0,\label{eq:linear2}
\end{eqnarray}
while  $z(t)=Z(t)v_0$ solves
\begin{eqnarray}
  z_t + \alpha^2Az_t + \nu ADz = Pf - P(Dv\cdot\nabla
  Dv), \label{eq:nonlinear1} 
  \\
  z(0)=0,\label{eq:nonlinear2}
\end{eqnarray}
with $v$ solution of~\eqref{eq:radm1}.  We take the $L^2$-inner product
of~\eqref{eq:linear1} with $y$ to find
\begin{equation*}
  \frac{d}{dt}\left( \norm{y(t) }^2 + \alpha^2 \norm{\nabla y(t)}^2
  \right) + 2\nu\norm{\nabla D^{1/2} y}^2 = 0, 
\end{equation*}
and then using that $D$ is positive and bounded from below, along with the
Poincar\'e inequality, we obtain as in the previous calculations,
%
\begin{equation*}
\frac{d}{dt}\left( \norm{y(t) }^2 + \alpha^2 \norm{\nabla y(t)}^2
\right) + c_0 \left( \lambda \norm{y(t)}^2 + \alpha^2 \norm{\nabla
    y(t)}^2   \right) \le 0,
\end{equation*}
which immediately implies
\begin{equation*}
\frac{d}{dt}\left( \norm{y(t) }^2 + \alpha^2 \norm{\nabla y(t)}^2 \right)
\le 
e^{-c_0 t}
\left( \norm{ v_0 }^2 + \alpha^2 \norm{\nabla v_0}^2 \right).
\end{equation*}
Therefore  the semigroup $Y(t)$ is exponentially contractive.

By using the properties of $D$, and standard inequalities in Sobolev
spaces (cf.~\cite[Eq.~(3.11)]{KT09}) it follows that
\begin{equation*}
  \begin{aligned}
    \| P( D v\cdot\nabla D v)\|_{V^{-1/2}}&=\sup_{\phi\in
      V^1,\,\|A^{1/4}\phi\|=1} \Big( v\cdot\nabla D v,P\phi\Big)
    \\
    &    \leq C\| Dv\|_{V^1}^2\|A^{1/4}\phi\|
    \\
    &\leq    C d_1^2\|v\|_{V^1}^2\|A^{1/4}\phi\|,
  \end{aligned}
\end{equation*}
hence showing that, since $v$ is a weak solution to~\eqref{eq:cadm1},
then 
\begin{equation*}
   P( D v\cdot\nabla D v)\in L^\infty(0,T;V^{-1/2}),
\end{equation*}
for all positive $T$, and with bounds independent of $T$, but
depending essentially on $\|D\|$.

Then we take~\eqref{eq:nonlinear1}-\eqref{eq:nonlinear2} and we test
it by $A^{1/2}z$ (procedure which can be again justified by
approximation with Faedo-Galerkin method). We observe that
$\|z\|_{V^s}=(A^sz,z)=\|A^{s/2}z\|^2$, by definition. Hence, by noting
that $Pf - P(Dv\cdot\nabla Dv)\in L^\infty(0,T;V^{-1/2})$ 
we get
\begin{equation*}
  \begin{aligned}
  &\frac{1}{2}\frac{d}{dt}\left(\alpha^2\|A^{3/4}z\|^2+\|A^{1/4}z\|^2\right)+\nu\|A^{3/4}D^{1/2}
  z\|^2
  \\
  &\quad =\frac{1}{2}\frac{d}{dt}\left(\alpha^2\|z\|^2_{V^{3/2}}+\|z\|^2_{V^{1/2}}\right)+\nu\|D^{1/2}
  z\|^2_{V^{3/2}}
  \\
  &  \qquad = (Pf - P(Dv\cdot\nabla Dv),A^{1/2}z)
  \\
  & \qquad = (A^{-1/4}(Pf - P(Dv\cdot\nabla Dv)),A^{3/4}z)
  \\
&\qquad \leq(\|f\|_{V^{-1/2}}+\|P(Dv\cdot\nabla Dv))\|_{V^{-1/2}})\|z\|_{V^{3/2}}.
\end{aligned}
\end{equation*}
Using previous estimates we obtain
\begin{equation*}
  \frac{d}{dt}\left(\alpha^2\|z\|^2_{V^{3/2}}+\|z\|^2_{V^{1/2}}\right)+\nu\|D^{1/2}
  z\|^2_{V^{3/2}}
   \leq(\frac{C}{\nu}\left(
\|f\|_{V^{-1/2}}^2+d_1^2\|v\|^4_{V^{1}}\right),
\end{equation*}
and the bound on $\|v\|_{V^1}\leq M_1$ implies that $z$ is bounded in
$V^{3/2}$.
  This means that $Z(t)$ maps $V^1$ into $V^{3/2}$,
and thanks to the compact embedding $V^{3/2}\subset V^1$, $Z(t)$ is
compact.  Thus the assumptions of Theorem~\ref{thm:attract} are satisfied,
and we conclude that $S(t)$ is asymptotically compact.  By~\cite{T97},
we conclude that $S(t)$ has a global attractor in $V^1$; the rest of
the proof follows as in~\cite{KT09}. With the same arguments one can
also show that if the force is smoother, say in $Pf\in V^1$, then the
attractor is a compact set in $V^2$.
\end{proof}
\section{Pipe flows}
\label{sec:pipe}
In this section we perform some analytical computations with very
classical tools, in order to focus on some special analytical features
of the solutions of Voigt models and also to better understand the
possible role of the RADM (and thus the NS-Voigt system for $D_0=I$)
in modeling of visco-elastic flows.  We recall that prior to turbulence
modeling, NS-Voigt equations appeared as a linear model for certain
non-Newtonian fluids, see a series of papers by
Oskolkov~\cite{Osk1973,Osk1982}

In order to understand, at least qualitatively, the effect of the
(non-viscous) damping term $-\alpha^{2} \Delta u_t$ on the behavior of
solutions, we restrict to the very simple setting of: a) fully
developed flows, that is the velocity is directed only along the axis
of the pipe; and b) we restrict to the NS-Voigt case; that is, assume
$D=D_0=I$.  The deconvolution order $N>0$ could be used, but details
would change for each specific deconvolution operator, and also the
analysis will become (even more) technical and also boundary
conditions for filtering at the wall would need to be specified.
Under the two above assumptions we develop analytical solutions, and
even if we are aware that these solutions are not turbulent and
describe only the linear behavior, they constitute a good source of
data for debugging complex codes, as only very few analytical
solutions are known to exist.  Moreover, analytical solutions can be
used to gain some physical insight into the model's solutions.  On the
other hand, we recall that a remarkable feature of the NS-Voigt
equations is that is one of the few LES models which can be fully
treated with Dirichlet boundary conditions and with given external
force, see~\cite{CLT2006,Larios-thesis}. Since the NS-Voigt equations
can be treated in the presence of solid boundaries with Dirichlet boundary
conditions, it is natural to investigate some class of exact solutions
for NS-Voigt, as in the case of the NSE, to directly compare the
results.
\subsection{Remarks on the regularity  of the solution}
\label{sec:regularity}
In this section we make some remarks on the solution of the
NS Voigt equations, but similar reasoning can be transferred
to the RADM (at least in the space periodic-setting, while here we
have a problem with solid boundaries and zero Dirichlet
conditions). In particular, as remarked in Larios~\cite{Larios-thesis}
the NS-Voigt have a pseudo-parabolic behavior, as the abstract
equations considered in Carroll and Showalter~\cite{CS1976}.  One of
the most interesting point is that the solution of the NS-Voigt
preserve the same regularity as the initial datum. Contrary to the
NSE for the solution of the NS-Voigt it is not
possible to show an improvement of the regularity, especially for what
concerns the space-variables. We recall that for the NSE (or more
generally parabolic equations), as soon as a weak solution is strong,
then it becomes smooth (even $C^\infty$ if the force and the boundary
of the domain are so) for all positive times. This has particular
consequences on the fine properties of weak solution of the NSE, when
obtained as limits of solutions to the NS-Voigt equations as
$\alpha\to0^+$, see~\cite{BS2015}.

Here we show with a particular simple example that the nonlinear
NS-Voigt cannot produce an improvement of the regularity, by making
some very explicit computations on the spectrum of the Laplace
operator. This can be also used to shed some light in the different
behavior of the NS-Voigt with respect to the NSE, concerning some very
special properties.

Let us consider the NS-Voigt system in a special physical
situation. We have a channel with cross section $E$ and with axis
directed along the $x_3$-direction.  We consider an incompressible
fluid, modeled by NS-Voigt equations in a semi-infinite straight pipe
$P=E\times \R^+\subset\R^3$. In a reference frame with $z$ (we call
$z:=x_3$ for historical reasons) directed along the pipe axis and
$x:=(x_1,x_2)$ belonging to an orthogonal plane, the NS-Voigt
equations read
\begin{equation}
  \label{eq:pipe-flow}
  \begin{aligned}
    \vec{v}_t-\alpha^2 \Delta\vec{v}_t+(\vec{v}\cdot\nabla)\,\vec{v}-\nu\Delta
    \vec{v}+\nabla p=\vec{f}&\qquad (x,z)\in E\times\R^+,\; t\in\R^+,
    \\
    \nabla\cdot \vec{v}=0&\qquad (x,z)\in E\times\R^+,\; t\in\R^+,
    \\
    \vec{v}=0&\qquad (x,z)\in \partial E\times\R^+,\; t\in\R^+,
  \end{aligned}
\end{equation}
where $\vec{v}(t,x,z)$ and $p(t,x,z)$ respectively denote velocity and
pressure, and $\nu>0$ represents kinematic viscosity.
\begin{remark}
  In this subsection and later on also in Section~\ref{sec:pulsatile}
  we use the symbol $\vec{v}(t,x,z)$ to remember that the velocity is
  a vector field and that we will look for special solutions, where
  the axial velocity $w(t,x,z)$ is the only non-zero one.
\end{remark}

By following classical calculations dating back to Hagen and
Poiseuille, we look for fully-developed solutions (also named
Poiseuille-type solutions):
\begin{equation*}
  \vec{v}(t,x,z) = (0,0,w(t,x))
    \qquad\text{and}\qquad
  p(t,x,z) = -\lambda(t,x,z)+p_0(t),
\end{equation*}
where $p_0(t)$ denotes an arbitrary function of time. The
Poiseuille-type \textit{ansatz} implies that the convective term
cancels out identically, and this motivates the statement that the
results will concern only the linear behavior of solutions. Moreover,
it is easy to deduce from the equations that the pressure is
$p(t,z)=-\lambda(t)\, z$. Finally, the dependence of $w$ on the space
variables $x_1$ and $x_2$ allows us to consider a problem reduced to
the cross-section $E$: In the so-called ``direct problem,'' without
external force, a given pressure-drop is assigned and the problem is the
following: Given $\lambda(t)$, find $w(t,x)$ such that
\begin{equation}
  \label{eq:basic_flow}
  \begin{cases}
    w_t(t,x) - \alpha^2\Delta_x w_t(t,x)-\nu\Delta_x w(t,x) =
    \lambda(t),&\qquad x\in 
    E,\; t\in\R^+,
    \\
    w(t,x)=0&\qquad x\in \partial E,\; t\in\R^+,
    \\
    w(0,x)=w_{0}(x)&\qquad x\in E,
  \end{cases}
\end{equation}
where $\Delta_x$ denotes the Laplacian with respect only to the variables
$x_1$ and $x_2$.  

In this particular setting the nonlinear term is identically
vanishing, but we still have a boundary value problem. We will start
by making some explicit calculations in the case $\lambda(t)$ is
extremely smooth, say $\lambda\equiv1$ and we assign an initial datum
$w_0(x)\in H^{1}_{0}(E)$. This corresponds to the classical pressure drop problem with
a constant gradient superimposed to sustain the flow, but as we will
see this assumption is completely inessential. We have the following
theorem.
\begin{theorem}
\label{thm:smoothness}
Let $E\subset\R^2$ be smooth and bounded. Let 
$\vec{v}(0,x,z)=(0,0,w_0(x))$ be given such that $w_0(x)\in H^1_0(E)$, and
let $\lambda=1$. Then,
 there exists a unique solution
to~\eqref{eq:basic_flow} such that
\begin{equation*}
  w\in L^\infty(0,T;H^1_0(E))\ \quad\forall\,T>0
\end{equation*}
hence a unique solution to the problem~\eqref{eq:pipe-flow} under the
Hagen-Poiseuille ansatz. Furthermore, if $w_0\not\in H^{s}(E)$ for any $s>1$, then
also $w(t,x)\not\in H^s(E)$ for all $t>0$.
\end{theorem}

The lack of regularization in the NS-Voigt equation is due to the special
role of the inviscid regularization and it has some consequences when
the Voigt model is used for theoretical purposes, see for instance~\cite{BS2015}.
\begin{proof}[\textit{of Thm.~\ref{thm:smoothness}}]
  First we note that the existence and uniqueness of solution with
  initial datum $w_0$ and external force $\vec{f}=(0,0,0)$ follows as
  in~\cite{CLT2006}. Here we make a more explicit construction in
  order to highlight the regularity properties. The proof is based on
  an explicit analysis of the frequency budget, working with a
  spectral basis to construct the solution. (In the space-periodic
  case this can be done even more explicit with a Fourier series
  expansion, but here we deal with the boundary value problem). 

  Let $\{\phi_m\}$ denote the eigenfunctions of the Laplace operator
  with Dirichlet conditions on $E$,
\begin{equation*}
  \begin{aligned}
    -\Delta \phi_m&=\lambda_m\phi_m \quad \text{in }E,
    \\
    \phi_m&=0\quad \text{on }\partial E,
  \end{aligned}
\end{equation*}
and by standard spectral theory the scalar functions $\{\phi_m\}$ are
smooth, orthonormal with respect to the $L^2(E)$ scalar product, and
the eigenvalues $\{\lambda_m\}$ are an increasing sequence of strictly
positive terms. Hence we construct the solution as
$w(t,x)=\sum_{m\geq1} c_m(t)\phi_m(x)$ and it is immediate to check
that we have to solve the following ordinary differential equation,
for all $m\geq1$
\begin{equation}
  \label{eq:ode}
  (1+\alpha^2 \lambda_m) c'_m(t)+\nu c_m(t)=\beta_m,
\end{equation}
where $\beta_m=\int_E \phi_m(x)\,dx$ and with the initial condition
$c_m(0)=\int_E w_0(x)\,\phi_m(x)\,dx$. Observe that, from the
regularity hypotheses on the initial datum, we have 
\begin{equation*}
  \sum_{m\geq1} \lambda_m |c_m(0)|^2<+\infty\qquad   \sum_{m\geq1}
  \lambda_m^s |c_m(0)|^2=+\infty,\text{ for all }s>1.
\end{equation*}
    Explicit integration  of~\eqref{eq:ode} gives 
    \begin{equation*}
      c_m(t)=c_m(0)\,e^{-\frac{\lambda_m \nu}{1+ \alpha^2\lambda_m
        }t}+\beta_m\left[\frac{1}{\lambda_m  \nu }-\frac{e^{-\frac{t
              \lambda_m  \nu}{1+\alpha^2\lambda_m  }}}{\lambda_m  \nu
        }\right]\qquad m\geq1.
    \end{equation*}
    It is straightforward to check that 
    $\sum_{1\leq m\leq N} c_m(t)\phi_m(x)\to w(t,x)$ as $N\to+\infty$,
    and from the explicit
    expression of the solution we can see that even if $\beta_m$ would
    be zero, then $\|w(t,x)\|_{H^s}=\infty$, for all $s>1$, and for
    all $t\geq0$, due to the
    fact that the exponential $e^{-\frac{\lambda_m \nu}{1+
        \alpha^2\lambda_m }t}$ does not produce any smoothing (as if
    $\alpha=0$), since
\begin{equation*}
\frac{\lambda_1 \nu}{1+ \alpha^2\lambda_1 }\leq  \frac{\lambda_m \nu}{1+
  \alpha^2\lambda_m }\leq\frac{\nu}{\alpha^2}.
\end{equation*}
The situation will be different for what concerns time
regularity, since any  time-differentiation will produce the zeroth
order term $\frac{\lambda_m \nu}{1+
  \alpha^2\lambda_m }$  which will not affect the convergence of the
series describing the solution.
\end{proof}
\subsection{Pulsatile flows}
\label{sec:pulsatile}
In this section we consider pulsatile flows, which are driven by a
time-periodic force (the pressure gradient). This is a remarkable
case motivated also by recent  studies in~\cite{Bei2005c,BGMS2014} for
heart-beat driven, human physiological flows, including blood and
cerebrospinal fluid.
In the so-called `direct problem,' a pulsatile pressure-drop is
assigned in problem~\eqref{eq:basic_flow}.  Together with this
problem, in many cases of physiological interest (when the measure of
the pressure cannot be done directly with sufficient accuracy by
phase-contrast Magnetic Resonance Imaging  or Doppler
ultrasonography) it will be also meaningful to study a problem with
assigned flux: $ \int_{E} w(t,x)\,dx=\Phi(t)$, for some smooth-enough
given scalar function $\Phi(t)$.  The Poiseuille-type \textit{ansatz}
will lead to the problem: Given $\Phi(t)$, find $(w(t,x),\lambda(t))$
such that
\begin{equation*}
  \begin{cases}
    \partial_t w(t,x) - \nu\Delta_x w(t,x) = \lambda(t),&\qquad x\in 
    E,\; t\in\R^+,
    \\
    w(t,x)=0,&\qquad x\in \partial E,\; t\in\R^+,
    \\
    \int_{E} w(t,x)\,dx=\Phi(t),&\qquad t\in\R^+,
  \end{cases}
\end{equation*}
which is an \textit{inverse} problem and it is linked to one of the
nowadays classical Leray problems. 

The study of this problem
when $E$ is a circle,
$\lambda(t)$ is time-periodic, and $\alpha=0$ dates back to Richardson
(1926) and Sexl (1930), while a big interest for applications to
bio-fluids raised especially after the work of the physiologist
J.~R.~Womersley~\cite{Wom1955}. In particular, he quantitatively clarified why for
the NSE, if the pulsation is fast enough, then the qualitative
behavior of solution can change drastically: Instead of a parabolic
type profile (as in the stationary case) a new phenomenon called
``\textit{{A}nnular effect}'' appears: If the frequency of pulsation
of $\lambda$ is large enough (with respect to the other relevant
parameters of the problem), then the maximum of the velocity is not
located along the axis of the channel, but near the wall.

This is particularly of interest in blood modeling, where --at first
approximation-- we have a pulsatile flow (blood driven by heart beat)
in a straight channel (larger arteries). This motivated the analysis
of Womersley, who explained by analytical formul\ae\ the experimental
observation that in larger arteries of the rabbit and the dog there is
also a reversal of the flow. This suggests that pulsatile flows, even
in the linear range, can be different and much more complicated than
the steady ones. The analysis of Womersley was based on a single mode
pressure-drop $\lambda(t)=e^{i n t}$ and a circular cross-section
$E=\{x^2_1+x^2_2<R\}$.  In cylindrical coordinates $(r,\theta)$ and
after the separation of variables $w(t,x)=e^{i\,nt}W(r)$, the
system~\eqref{eq:basic_flow} becomes: find $W(r)$ such that
\begin{equation*}
  \begin{aligned}
    i\,n\, W(r) -\nu\Big(W''(r)+\frac{W(r)}{r}\Big)&=1\qquad\text{for
    } 0<r<R, 
    \\
    W'(0)=W(R)&=0.
  \end{aligned}
\end{equation*}
The above equation for a cylindrical symmetric solution can be solved
by means of an expression involving the zeroth order Bessel function
$J_0(r)$. The most widely known contribution\footnote{Most important to this
  field, but recall that he recruited A.M.~Turing for the
  code-breaking project in WWII. He also convinced Sir G. Darwin to
  start the project of an \textit{all British} computer and coined the
  name Automatic Computing Engine (ACE) for an early electronic
  computer.} of Womersley is that of determining the non-dimensional
quantity (named later Womersley number)
\begin{equation*}
  \text{Wo}:={R}\sqrt{\frac{\omega}{\nu}},
\end{equation*}
which determines the qualitative behavior of solutions, where $R$ is
the radius of the channel, $\omega$ the frequency of the pressure
drop, and $\nu$ the kinematic viscosity.

The Womersley number is the ratio of the transient or oscillatory
inertia force to the shear force. When $\text{Wo}\leq 1$, the
frequency of pulsations is sufficiently low that a parabolic velocity
profile has time to develop during each cycle, and the flow will be
very nearly in phase with the pressure gradient. In this case using
instantaneous pressure to modulate a parabolic profile will give a
reasonable solution. When $\text{Wo}\geq10$ the frequency of
pulsations is sufficiently large that the velocity profile is
relatively flat or plug-like, and the mean flow lags the pressure
gradient by about 90 degrees.

In addition to this peculiar effect, the Womersley solution has become
a paradigm in the analysis of biological fluids (see
Fung~\cite{Fun1997} and Quarteroni and Formaggia~\cite{QF2004}) for
its simplicity, being one of the very few exact unsteady solutions to
the NSE, to debug complicated CFD flow, and also to provide
an improved source of initial/boundary data, as well as a benchmark
solution for pulsatile flows.

We now perform an analysis of the NS-Voigt equations for this problem. Assume that  
$\Phi(t)=e^{i\,\omega t}$, and by the ansatz $w(t,x)=e^{i\omega t}W(r)$ in the case
of~\eqref{eq:cadm1} we arrive at the following boundary value problem for
a single ODE: Find $W(r)$ such that
\begin{equation*}
  \begin{aligned}
    i\,\omega\,\Big[
    W(r)-\alpha^{2}\Big(W''(r)+\frac{W(r)}{r}\Big)\Big]
    -\nu\Big(W''(r)+\frac{W(r)}{r}\Big)&=1\qquad\text{for } 0<r<R, 
    \\
    W'(0)=W(R)&=0.
  \end{aligned}
\end{equation*}
The above equation can be exactly solved by means of an expression
involving the zeroth order Bessel function $J_0(r)$.  By explicit
calculations, we get
\begin{equation*}
  W(r)=\frac{1}{i\, \omega} \left(1-\frac{J_0\left(\frac{i r
          \sqrt{\omega }}{\sqrt{\alpha ^2 \omega 
            -i \nu }}\right)}{J_0\left(\frac{i R \sqrt{\omega }}{\sqrt{\alpha ^2 \omega
            -i \nu }}\right)}\right),
\end{equation*}
which reduces to the classical solution obtained by Womersley in the
case $\alpha=0$. Hence, we get that the natural non-dimensional
quantity is the following (which we call $\alpha$-Womersley number)
\begin{equation*}
  \alpha\text{-Wo}
  :=R\,\frac{\sqrt{\omega}}{\sqrt[4]{\alpha^{4}\omega^{2}+\nu^{2}}}.
\end{equation*}
The $\alpha$-Womersley number is always smaller than the Womersley
number of the flow with the same viscosity, radius of the pipe, and
frequency. In fact
\begin{equation*}
  \frac{\text{Wo}}{\alpha\text{-Wo}}=
\sqrt[4]{1+\frac{\alpha^{4}\omega^{2}}{\nu^{2}}}, 
\end{equation*}
and the difference can be significant if $\alpha$ is not properly
chosen. This gives the hint that a not finely tuned choice of $\alpha$ can
completely change the behavior of solutions, since the flow can pass
from one regime to another, if $\text{Wo}$ is large and $\alpha$ is
not small enough.

Here, we want to give a more explicit and real solution by considering
the 2D case, that is, the cross section $E$ is the 1D interval
$E=:]-R,R[$, and showing the effect of a pulsatile force $\cos(\omega
t)$, which is not a peculiar effect of the space dimension. In
particular, we will see that the qualitative behavior is not better in
a 2D pipe, with respect to the 3D one. We consider then the following
toy problem
\begin{equation*}
  \begin{cases}
    \partial_t w(t,x) - \alpha^2\,w_{txx}(t,x)-\nu\, w_{xx}(t,x) =
    \cos(\omega t),&\qquad x\in]-R,R[ ,\; t\in\R^+,
    \\
    w(t,x)=0&\qquad x=\pm R,\; t\in\R^+,
  \end{cases}
\end{equation*}
which corresponds to the pulsatile flow in a rectangular 2D channel.
In this case the exact solution (which is real as can be checked by
direct inspection) has the following explicit expression:
\begin{equation*}
  \begin{aligned}
    w(t,x)=\frac{\sin(\omega\,t)}{\omega}+\frac{1}{2\omega}(-\sin
    (\omega\,t )-i \cos (\omega\,t )) \,
    \text{sech}\left(\frac{R}{\sqrt{\alpha ^2+\frac{i \nu }{\omega
          }}}\right) \cosh \left(\frac{x}{\sqrt{\alpha ^2+\frac{i \nu
          }{\omega }}}\right)
    \\
    +\frac{1}{2\omega} (- \sin (t \omega )+i \cos (\omega\,t ))
    \,\text{sech}\left(\frac{R}{\sqrt{\alpha ^2-\frac{i \nu }{\omega
          }}}\right) \cosh \left(\frac{x}{\sqrt{\alpha ^2-\frac{i \nu
          }{\omega }}}\right).
  \end{aligned}
\end{equation*}
If we set, in the case $\alpha=0$ (i.e. the NSE), $\nu=1$, $R=1$,
and $\omega=144$,
we find Wo=12,
and the solution at $t=0$ shows the annular effect and reversal of the
flow at distance about $R/2$, as shown in Figure~\ref{Wo12}.

\begin{figure}[!h]
\begin{center}
\epsfig{file=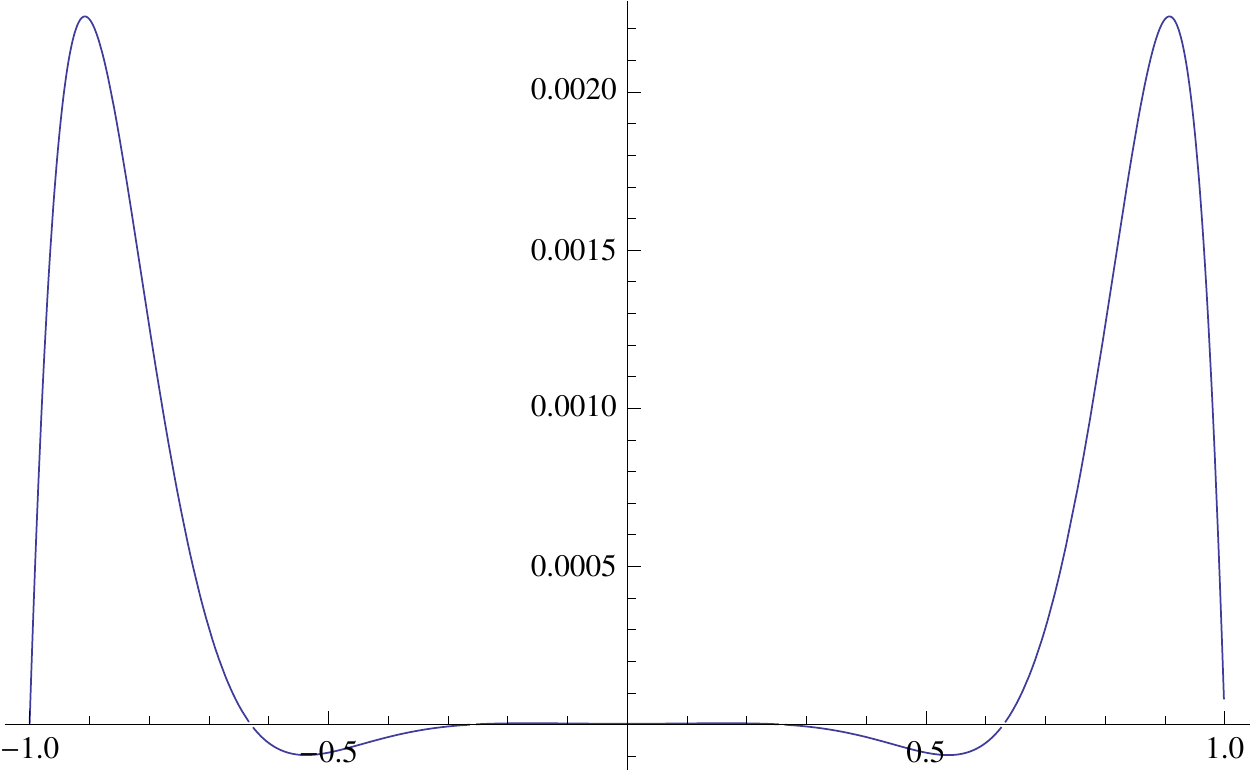,width=6.0cm}
 \end{center}
 \caption{\label{Wo12} Shown above is the profile of the Navier-Stokes
   solution at time $t=0$ with $\textrm{Wo}=12$.}
\end{figure}
Next, we plot with the same parameters, the corresponding solutions in
terms of different values of $\alpha$.
\begin{figure}[h]
\label{fig:annular-Voigt}
\centering
            \includegraphics[height=3.5cm]{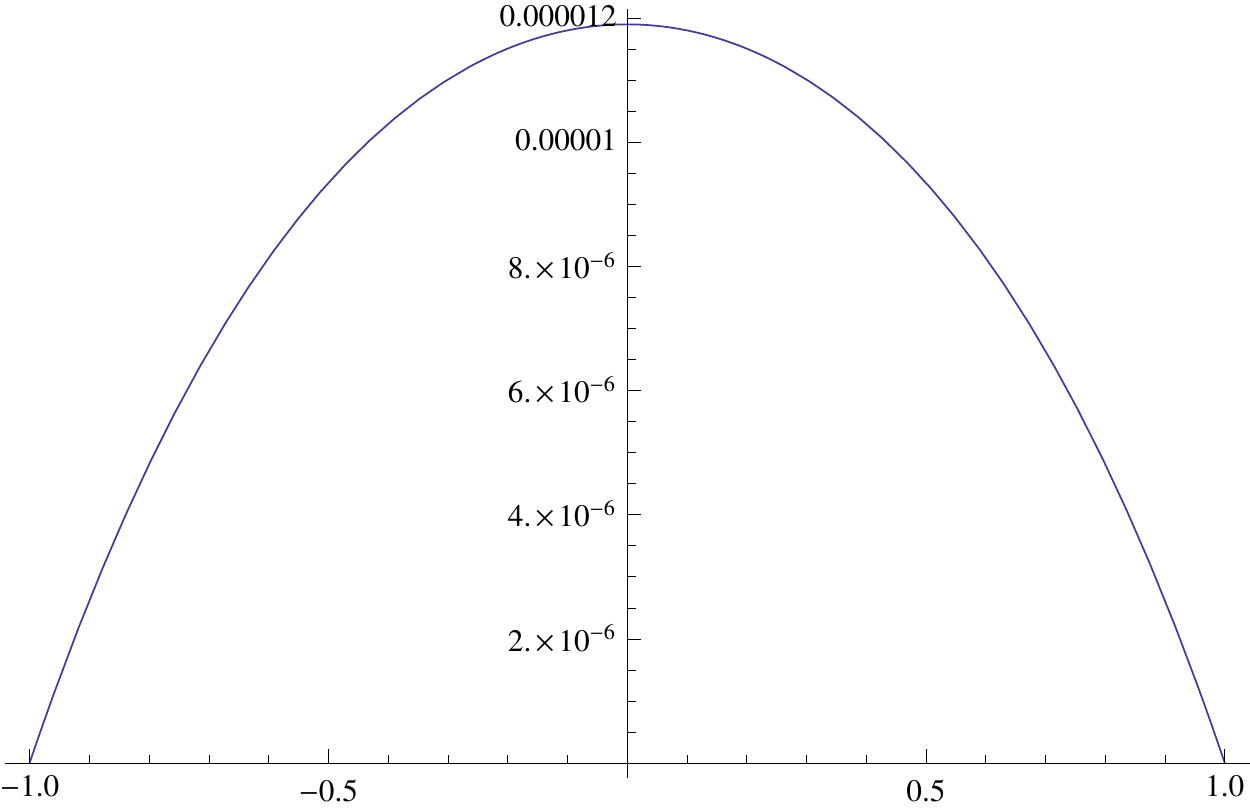}
\hspace*{1cm}\includegraphics[height=3.5cm]{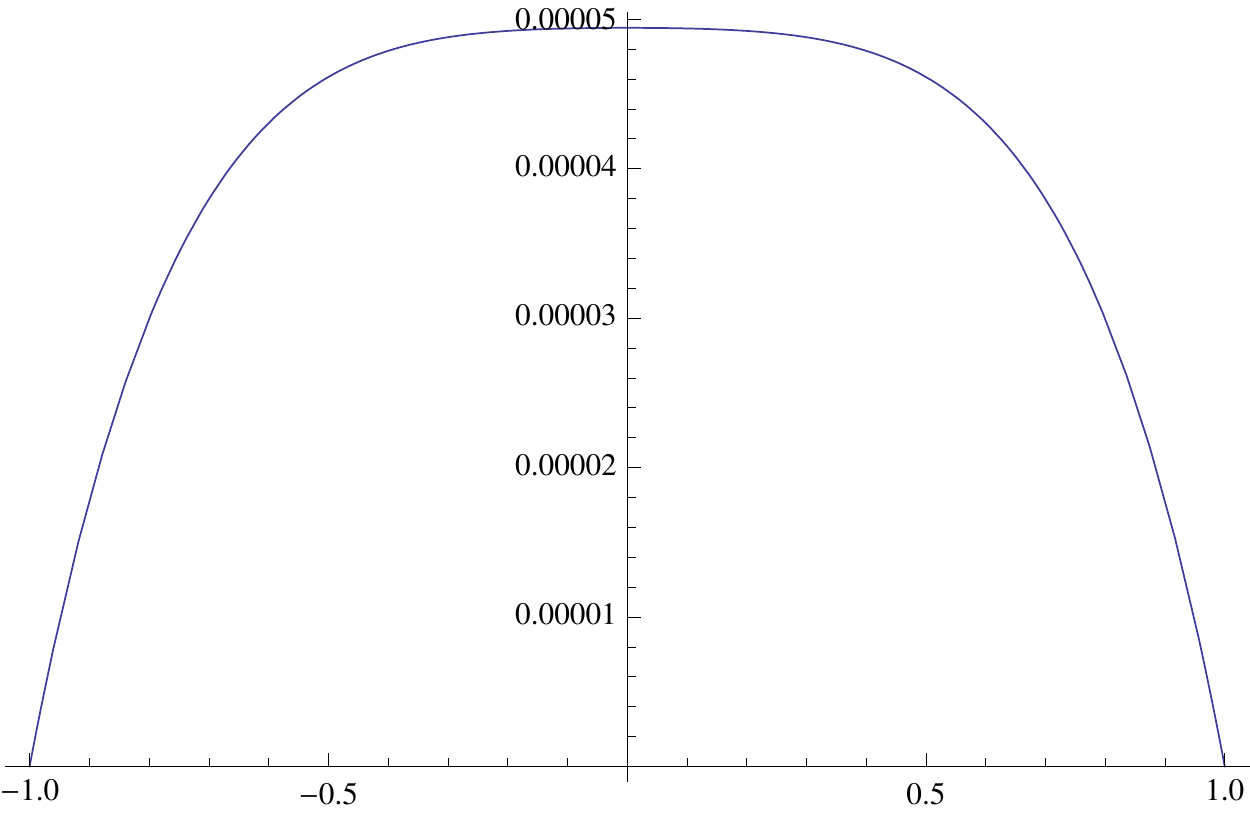}

\vspace*{1cm}

            \includegraphics[height=3.5cm]{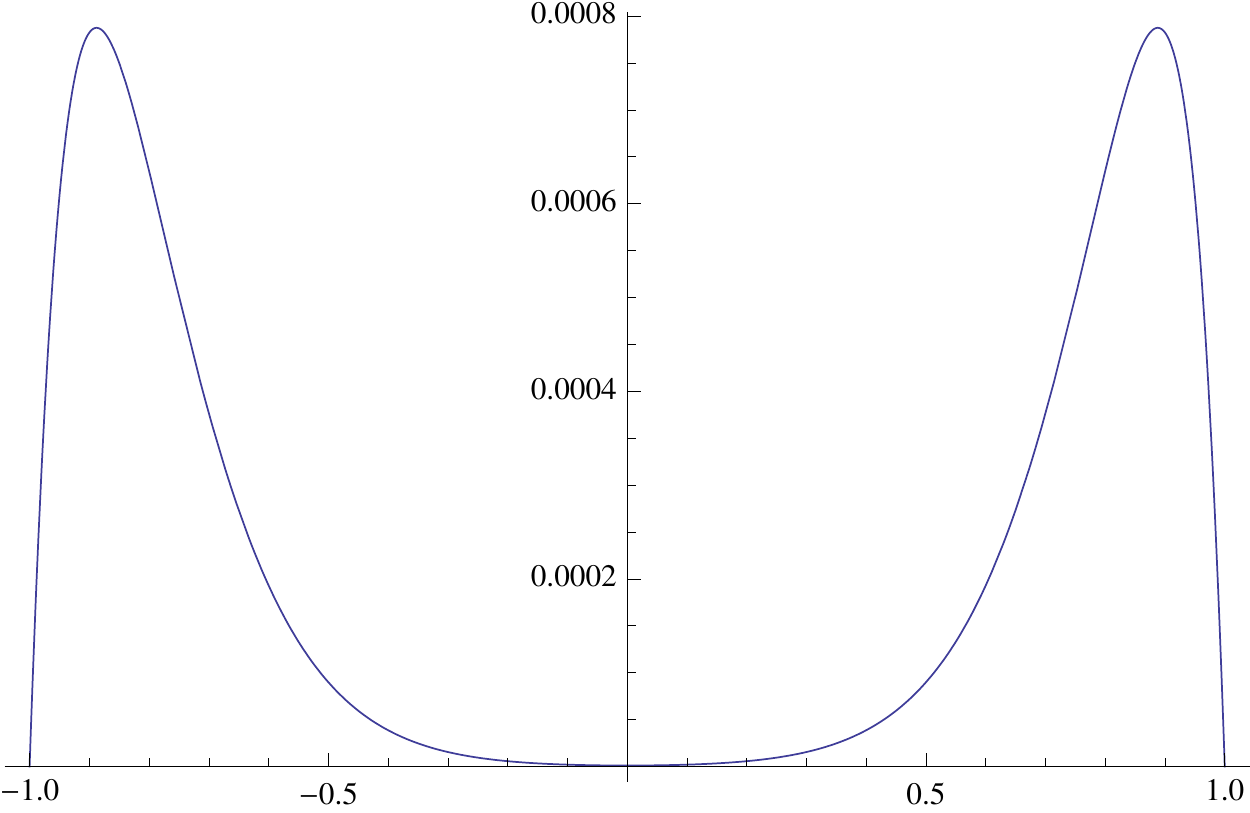}
\hspace*{1cm}\includegraphics[height=3.5cm]{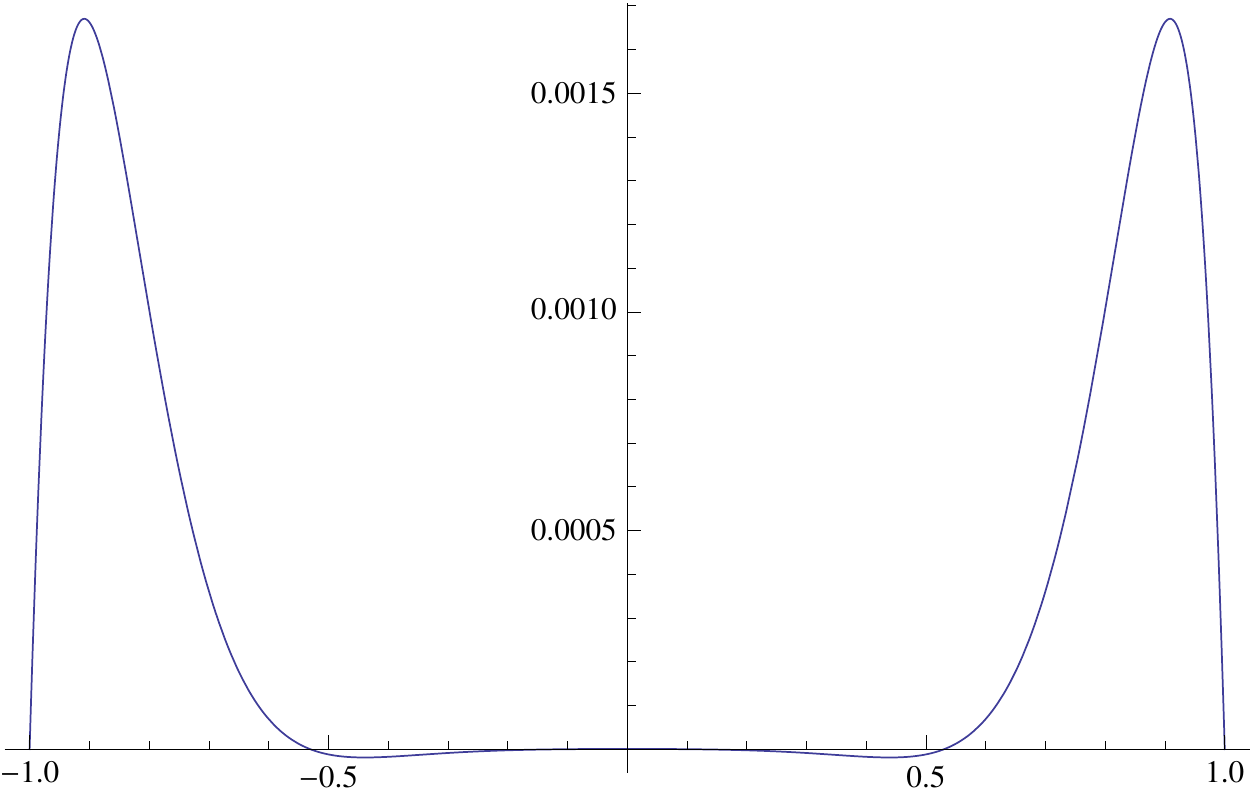}
   \caption{Profile of the solution at time $t=0$ with (a) $\alpha=1$, hence
  $\alpha\text{-Wo}=0.999988$; (b) $\alpha=1$, hence
  $\alpha\text{-Wo}=1.99961$; (c) $\alpha=0.1$, hence
  $\alpha\text{-Wo}=9.06295$ and still no reversal starts (even if
  the solution 
  at $x=0$ is about $10^{-6}$); (d) $\alpha=0.1$, hence
  $\alpha\text{-Wo}=11.6399$}
\end{figure}
Pictures in Figure~\ref{fig:annular-Voigt} show that the qualitative
behavior is quite different with respect to the variations of
$\alpha$. Simple calculations show that, in order to reconstruct the
correct pulsatile behavior with flow reversal, the parameter $\alpha$
must satisfy
\begin{equation*}
  \alpha<\frac{\sqrt{\nu } \sqrt[4]{\text{Wo}^4-10000}}{10 \sqrt{\omega }}.
\end{equation*}
As is evidenced in Figure~\ref{fig:annular-Voigt}, this puts strong
limitations on the choice of $\alpha$, and moreover such a small
$\alpha$ moves the NS-Voigt systems closer and closer to the
NSE. Moreover, in LES, $\alpha$ has the role of the smallest resolved
length scale, and when looking at the system in this sense, the
limitations for this model are those of a good resolution of the flow,
otherwise transient features are lost. This could be used to speculate
that the space filter alone cannot be very good when studying
time-evolution problems with transient or periodic effects, or at the
very least the smallest resolved length scale should not be constant
across the domain.  This idea is mentioned in~\cite{GRT13}, and indeed
better results for turbulent channel flow were found when $\alpha$ was
chosen of the order of the local mesh width, which was much finer in
the boundary layer.
\section{Conclusions}
We have presented a detailed mathematical analysis of several
important aspects of the RADM.  In particular, we presented
a detailed analysis of the model's treatment of energy, which showed
that deconvolution acts to drain energy from the system, and likewise
affects the energy spectrum.  We derived a scaling for energy in its
Fourier modes, which shows that increasing deconvolution in the model
shortens the inertial range, and that the inertial range is split into
two parts: the larger scales of the inertial range have a $k^{-5/3}$
scaling with energy, while the lower numbers have a faster $k^{-3}$
scaling.  Interestingly, this is the same spectral scaling as both the
Leray-$\alpha$ and simplified Bardina models.  We presented energy
spectra from computations of forced, isotropic, homogeneous turbulence
with varying $N$ and $\alpha$, which verified the predicted scalings,
and showed good agreement on large scales with DNS of the NSE.  In
addition to the RADM energy study, we also proved existence of smooth
global attractors, and provided an analytical study of pulsatile flows
that shows small $\alpha$ may be requested to study certain flows.

There are several important directions for future work with the RADM.
First and foremost, more benchmark testing needs done.  In particular,
since the model performed well with $Re_{\tau}=180$ turbulent channel
flow, future testing on the benchmarks of $Re_{\tau}$=395 and 590 are
natural next steps.  Turbulent flow around obstacles is another
important test problem, as is applying the model to specific
application problems such as flow through an aorta.  The model could
also be studied with different types of filtering and deconvolution.
%

\end{document}